\documentclass[11pt]{amsart}

\usepackage{fullpage}
\usepackage{amssymb}
\usepackage[table]{xcolor} 

\theoremstyle{plain}
\newtheorem{theorem}{Theorem}[section]

\newtheorem{conjecture}[theorem]{Conjecture}

\newtheorem{corollary}[theorem]{Corollary}

\newtheorem{lemma}[theorem]{Lemma}

\newtheorem{problem}[theorem]{Problem}
\newtheorem{proposition}[theorem]{Proposition}
\newtheorem{remark}[theorem]{Remark}

\numberwithin{equation}{section}

\newcommand{\aut}[1]{\mathrm{Aut}(#1)}
\newcommand{\atp}[1]{\mathrm{Atp}(#1)}

\newcommand{\rmlt}[1]{\mathrm{Mlt}_r(#1)}

\newcommand{\rinn}[1]{\mathrm{Inn}_r(#1)}
\newcommand{\hol}[1]{\mathrm{Hol}(#1)}
\newcommand{\lnuc}[1]{\mathrm{Nuc}_\ell(#1)}
\newcommand{\rnuc}[1]{\mathrm{Nuc}_r(#1)}
\newcommand{\mnuc}[1]{\mathrm{Nuc}_m(#1)}
\renewcommand{\ker}[1]{\mathrm{Ker}(#1)}
\newcommand{\setof}[2]{\{#1\mid #2 \}}
\newcommand{\genof}[2]{\langle{#1}\mid {#2}\rangle}
\newcommand{\F}{\mathbb{F}}
\newcommand{\Z}{\mathbb{Z}}

\begin{document}

\title{Bol loops and Bruck loops of order $pq$}

\author[M.~K.~Kinyon]{Michael K.~Kinyon}

\address[Kinyon, Vojt\v{e}chovsk\'y]{Department of Mathematics, University of Denver, 2280 S Vine St, Denver, CO 80208, USA}

\author[G.~P.~Nagy]{G\'abor P.~Nagy}

\address[Nagy]{Bolyai Institute, University of Szeged, Aradi v\'{e}rtan\'{u}k tere 1, H-6720 Sze\-ged, Hungary}

\address[Nagy]{MTA-ELTE Geometric and Algebraic Combinatorics Research Group, P\'azm\'any P. s\'et\'any 1/c, H-1117 Budapest, Hungary}

\author[P.~Vojt\v{e}chovsk\'{y}]{Petr Vojt\v{e}chovsk\'{y}}

\email[Kinyon]{mkinyon@du.edu}

\email[Nagy]{nagyg@math.u-szeged.hu}

\email[Vojt\v{e}chovsk\'{y}]{petr@math.du.edu}

\begin{abstract}
Right Bol loops are loops satisfying the identity $((zx)y)x = z((xy)x)$, and right Bruck loops are right Bol loops satisfying the identity $(xy)^{-1} = x^{-1}y^{-1}$. Let $p$ and $q$ be odd primes such that $p>q$. Advancing the research program of Niederreiter and Robinson from $1981$, we classify right Bol loops of order $pq$. When $q$ does not divide $p^2-1$, the only right Bol loop of order $pq$ is the cyclic group of order $pq$. When $q$ divides $p^2-1$, there are precisely $(p-q+4)/2$ right Bol loops of order $pq$ up to isomorphism, including a unique nonassociative right Bruck loop $B_{p,q}$ of order $pq$.

Let $Q$ be a nonassociative right Bol loop of order $pq$. We prove that the right nucleus of $Q$ is trivial, the left nucleus of $Q$ is normal and is equal to the unique subloop of order $p$ in $Q$, and the right multiplication group of $Q$ has order $p^2q$ or $p^3q$. When $Q=B_{p,q}$, the right multiplication group of $Q$ is isomorphic to the semidirect product of $\Z_p\times \Z_p$ with $\Z_q$. Finally, we offer computational results as to the number of right Bol loops of order $pq$ up to isotopy.
\end{abstract}

\thanks{Michael K.~Kinyon partially supported by Simons Foundation Collaboration Grant 359872 and FCT project CEMAT-CI\^{E}NCIAS UID/Multi/04621/2013. G\'abor P. Nagy's travel partially supported by Simons Foundation Collaboration Grant 210176 to Petr Vojt\v{e}chovsk\'y, and by the Humboldt return fellowship of A. Mar\'oti.}

\keywords{Bol loop, Bruck loop, K-loop, Bol loop of order $pq$, Bruck loop of order $pq$, multiplication group, nucleus, dihedral group, twisted subgroup, eigenvalue of a circulant matrix, quadratic field extension}

\subjclass[2010]{Primary: 20N05; Secondary: 12F05, 15B05, 15B33, 20D20.}

\maketitle

\section{Introduction}

Throughout the paper let $p$, $q$ be odd primes such that $p>q$.

Groups of order $pq$ are very well understood. By the Sylow theorems, any group $G$ of order $p^a q$ possesses a unique normal subgroup $P$ of order $p^a$, and is a semidirect product of $P$ with the cyclic group $\Z_q$. When $q$ does not divide $p^a - 1$ then $G$ also possesses a normal subgroup of order $q$ and $G\cong P\times \Z_q$. When $G$ has order $pq$ then either $G\cong \mathbb Z_{pq}$ or $q$ divides $p-1$ and $G$ is the unique nonabelian group of order $pq$ (cf. \cite[Section 4.4]{Hall_1959}).

In this paper we classify right Bol loops and right Bruck loops of order $pq$ up to isomorphism, generalizing the above result for groups. For the convenience of the reader, we summarize our main results in the following theorem:

\begin{theorem}\label{Th:Main}
Let $p>q$ be odd primes.
\begin{enumerate}
\item[(i)] A nonassociative right Bol loop of order $pq$ exists if and only if $q$ divides $p^2-1$.
\item[(ii)] If $q$ divides $p^2-1$, there exists a unique nonassociative right Bruck loop $B_{p,q}$ of order $pq$ up to isomorphism, and there are precisely $(p-q+4)/2$ right Bol loops of order $pq$ up to isomorphism.
\item[(iii)] If $q$ divides $p^2-1$, the $(p-q+4)/2$ right Bol loops of order $pq$ can be constructed on $\F_q\times \F_p$ with multiplication
    \[
        (i,\,j)(k,\,\ell) = (i+k,\, \ell(1+\theta_k)^{-1} + (j+\ell(1+\theta_k)^{-1})\theta_i^{-1}\theta_{i+k}),
    \]
    where $\theta_0$, $\dots$, $\theta_{q-1}\in\F_p$ are chosen as follows.

    Fix a non-square $t$ of $\F_p$, write $\F_{p^2} = \setof{u+v\sqrt{t}}{u,\,v\in\F_p}$, and let $\omega\in\F_{p^2}$ be a  primitive $q$th root of unity.  Let
    \[
        \Gamma = \left\{\begin{array}{ll}
             \setof{\gamma\in \F_p}{1\le \gamma\le (p+1)/2,\,1-\gamma^{-1}\not\in\langle\omega\rangle},&\text{if $q$ divides $p-1$},\\
             \setof{\gamma = 1/2+m\sqrt{t}}{0\le m\le (p-1)/2,\,1-\gamma^{-1}\not\in\langle\omega\rangle},&\text{if $q$ divides $p+1$},
        \end{array}\right.
    \]
    be a set of cardinality $(p-q+2)/2$. Then either let $\theta_i=1$ for every $i\in \F_q$, or choose $\gamma\in\Gamma$ and let $\theta_i = (\gamma\omega^i+(1-\gamma)\omega^{-i})^{-1}\in \F_p$ for every $i\in \F_q$.

    The choice $\theta_i=1$ for all $i$ results in the cyclic group of order $pq$. The choice $\gamma=1/2=(p+1)/2$ results in the nonassociative right Bruck loop $B_{p,q}$. If $q$ divides $p-1$, the choice $\gamma=1$ results in the nonabelian group of order $pq$.
\item[(iv)] Let $Q$ be a nonassociative right Bol loop of order $pq$. Then $Q$ contains a unique subloop of order $p$ and this subloop is normal and equal to the left nucleus of $Q$. The right nucleus and the middle nucleus of $Q$ are trivial. The right multiplication group of $Q$ has order $p^2q$ or $p^3q$.
\item[(v)] The right multiplication group of $B_{p,q}$ is isomorphic to $(\Z_p\times \Z_p)\rtimes \Z_q$.
\end{enumerate}
\end{theorem}

The nonassociative right Bol loops of order $pq$ will be available in the next release of the \texttt{LOOPS} package \cite{LOOPS} for \texttt{GAP} \cite{GAP}.

\subsection{Related results}

Bol loops were introduced in 1937 by G.~Bol in \cite{Bol}, where he studied the associated $3$-nets. The first systematic algebraic study of Bol loops is due to D.~Robinson \cite{Robinson_thesis,Robinson}, where he showed, among other results, that right Bol loops are right power alternative, and hence that any right Bol loop of prime order is a group.

R.~Burn proved in 1978 \cite{BurnI} that right Bol loops of order $p^2$ and $2p$ are groups, and classified nonassociative right Bol loops of order $8$.
In a seminal 1981 paper \cite{NiederreiterRobinson_pq}, H.~Niederreiter and K.~Robinson established a number of results for right Bol loops of order $pq$ and came close to classifying right Bol loops of order $3p$:

\begin{theorem}[\cite{NiederreiterRobinson_pq}]
Let $p>q$ be odd primes.
\begin{itemize}
\item If $q$ divides $p^2-1$ then there exists a nonassociative right Bruck loop $B_{p,q}$ of order $pq$, and a non-Bruck right Bol loop of order $pq$.
\item A right Bol loop of order $pq$ contains a unique subloop of order $p$, and when $q=3$ then the unique subloop of order $p$ is normal.
\item There are at least $(p+1)/2$ right Bol loops of order $3p$ up to isomorphism, and at least $(p+5)/6$ right Bol loops of order $3p$ up to isotopism.
\end{itemize}
\end{theorem}

They also showed that any right Bol loop of order $pq$ can be constructed from an ensemble of $q$ complete mappings on $\Z_p$. They were not able to establish that the unique subloop of order $p$ is normal even when $q>3$, nor that the complete mappings must be linear in order to yield a right Bol loop. Nevertheless, they obtained additional results (see below) under the \emph{assumption} that the subloop of order $p$ is normal and that the complete mappings are linear.

B.~Sharma and A.~Solarin came up with a conflicting estimate on the number of right Bol loops of order $3p$ \cite{SharmaSolarin_3p} but a problem with their proof was pointed out in \cite{NiederreiterRobinson_3p}. B.~Sharma also attempted to prove that the subloop of order $p$ in a right Bol loop of order $pq$ is normal \cite{Sharma_normal}, and that a right Bol loop of order $pq$ must be associative when $q$ does not divide $p^2-1$ \cite{Sharma_pq}. Both of these results turn out to be true---as we shall see---but the proofs in \cite{Sharma_normal,Sharma_pq} are incorrect (there are counterexamples to some intermediate claims made in the proofs).

R.~Burn went on to prove that there exist nonassociative right Bol loops of order $4n$ \cite{BurnII}. Moreover, for every odd prime $p$, there are precisely two nonassociative right Bol loops of order $2p^2$ up to isomorphism. (The history of this result is also convoluted: R.~Burn claimed in \cite{BurnIII} that there is a unique nonassociative right Bol loop of order $2p^2$, B.~Sharma constructed two examples of order $18$ \cite{Sharma_18}, R.~Burn accounted for the second class of examples in a correction to \cite{BurnIII}, and B.~Sharma and A.~Solarin gave an independent proof in \cite{SharmaSolarin_2p2}.)

We note that for Moufang loops, which are properly contained between groups and right Bol loops, the Moufang theorem \cite{Moufang} guarantees that every Moufang loop of order $pq$ is a group.

\medskip

From a more general perspective, Bol loops find applications in differential geometry (see the monograph \cite{Sabinin} and the references therein). Bruck loops appear naturally in the special theory of relativity. A.~Ungar showed in \cite{Ungar} that Einsten's relativistic addition of vectors gives rise to a nonassociative Bruck loop.

L.~Paige discovered an infinite family of nonassociative finite simple Moufang loops \cite{Paige}, and M.~Liebeck proved that no other nonassociative finite simple Moufang loops exist \cite{Liebeck}. Although examples of finite simple non-Moufang Bol loops were not easy to find, they abound \cite{Nagy_simple} even in the more restrictive case of Bruck loops \cite{BaumeisterStein, Nagy_simple2}. There are finite simple Bol loops of odd order $p^aq^b$ \cite{Nagy_simple}, for instance.

\subsection{Outline of the paper}

There are three main techniques present in this paper.

Standard loop-theoretical arguments (including strong results of Glauberman on right Bruck loops) provide necessary conditions on local properties of right Bol loops of order $pq$, such as the multiplication formula \eqref{Eq:Mult}, but also a proof that the unique subloop of order $p$ is normal and that the right and middle nuclei are trivial.

Group-theoretical arguments (mostly about groups of order $2p^aq^b$) shed light on the structure of the right multiplication groups of right Bol loops of order $pq$ and completely settle the structure in the case of right Bruck loops. As a consequence we deduce the global condition that nonassociative right Bol loops $Q$ of order $pq$ exist only if $q$ divides $p^2-1$, and also the fact that the left nucleus of $Q$ is of order $p$. From this it then follows easily that the complete mappings of \eqref{Eq:Mult} must be linear, a key step.

In the linear (and hence general) case, the isomorphism problem reduces to a solution of the bi-infinite recurrence relation \eqref{Eq:Recurrence} over $\F_p$ with period $q$, and to a classification of the solutions modulo the equivalence \eqref{Eq:Equivalence}. This was shown already in \cite{NiederreiterRobinson_pq}. We solve the recurrence relation by solving the eigenvalue/eigenvector problem for the associated circulant matrix over $\F_{p^2}$, mimicking the standard approach to circulant matrices over complex numbers. The difficulty lies in identifying the solutions with all entries in $\F_p\setminus\{0,-1\}$ (rather than in $\F_{p^2}$). The equivalence classes of solutions are then described and counted by elementary calculations in $\F_{p^2}$.

One of our goals was to present the various topics with approximately the same level of detail, so that the paper can be read by researchers who are not experts in all three areas (loop theory, group theory, finite fields).

\section{Preliminaries on Bol loops and Bruck loops}

We apply maps to the right of their arguments, and we conjugate by $u^v = v^{-1}uv$. In nonassociative situations, we use the dot convention to indicate the order of multiplications. For instance, $uv\cdot w$ stands for $(uv)w$.

\subsection{Basic properties of Bol loops and Bruck loops}

A nonempty set $Q$ with a binary operation $\cdot$ is a $\emph{loop}$ if all left translations and all right translations
\[
    L_u:Q\to Q,\ vL_u = uv,\quad\quad R_u:Q\to Q,\ vR_u = vu
\]
are bijections of $Q$, and if there is an identity element $1\in Q$ satisfying $1u=u1=u$ for every $u\in Q$. In a loop, we write $u\backslash v = vL_u^{-1}$ and $u/v = vR_u^{-1}$ for the left and right divisions, respectively.

For a loop $Q$, the \emph{right multiplication group} of $Q$ is the group
\[
    \rmlt{Q} = \genof{R_u}{u\in Q},
\]
and the \emph{right inner mapping group} of $Q$ is defined by
\[
    \rinn{Q} = \setof{\varphi\in\rmlt{Q}}{1\varphi=1}.
\]
The \emph{right section} of $Q$ is the set
\[
    R_Q=\setof{R_u}{u\in Q}.
\]
The right section $R_Q$ is a transversal (both left and right) to $\rinn{Q}$ in $\rmlt{Q}$, cf. \cite{Bruck}. In particular, every $\varphi\in\rmlt{Q}$ can be written uniquely as $\varphi = \psi R_u$, where $\psi\in\rinn{Q}$ and $u\in Q$.

A subgroup $H$ of a group $G$ is \emph{core-free} in $G$ if $H$ contains no nontrivial subgroups normal in $G$. It is well known, cf. \cite{Bruck}, that for any loop $Q$, $\rinn{Q}$ is a core-free subgroup of $\rmlt{Q}$.

For a loop $Q$ define the \emph{left nucleus}, \emph{middle nucleus} and \emph{right nucleus} by
\begin{align*}
    \lnuc{Q} &= \setof{u\in Q}{u(vw)=(uv)w\text{ for all }v,\,w\in Q},\\
    \mnuc{Q} &= \setof{v\in Q}{u(vw)=(uv)w\text{ for all }u,\,w\in Q},\\
    \rnuc{Q} &= \setof{w\in Q}{u(vw)=(uv)w\text{ for all }u,\,v\in Q},
\end{align*}
respectively. Each of the three nuclei is a subloop of $Q$, not necessarily normal in $Q$. Note that $\lnuc{Q} = \setof{u\in Q}{u\varphi = u\text{ for every }\varphi\in\rinn{Q}}$, and dually for the right nucleus.

A loop $Q$ is \emph{right Bol} if it satisfies the \emph{right Bol identity}
\begin{equation}\label{Eq:RBol}
    ((wu)v)u = w((uv)u).\tag{Bol${}_r$}
\end{equation}
The identity \eqref{Eq:RBol} can be restated as an identity for right translations, namely
\[
    R_uR_vR_u = R_{(uv)u}.
\]

It is well known, cf. \cite{Robinson_thesis}, that every right Bol loop $Q$ is \emph{power associative} (that is, every element generates an associative subloop), has the \emph{right inverse property} (that is, $uv\cdot v^{-1}=u$ for every $u$, $v\in Q$) and, more generally, is \emph{right power alternative} (that is, $R_u^i = R_{u^i}$ holds for every $u\in Q$ and $i\in\mathbb Z$). Consequently, $|R_u|=|u|$ for every $u\in Q$, and if $|Q|$ is finite then $|u|$ divides $|Q|$.

The right nucleus coincides with the middle nucleus in every right inverse property loop \cite{Bruck}. In a right Bol loop $Q$, the right nucleus is normal in $Q$ \cite[Lemma 2.1]{Nagy_invariants}.

Finally, in a right Bol loop, the left division can be expressed in terms of the multiplication and inverses by
\begin{equation}\label{Eq:LeftDivision}
    u\backslash v = (u^{-1}\cdot vu)u^{-1},
\end{equation}
cf. \cite[Lemma 2]{GlaubermanI}. Consequently, a nonempty subset of a right Bol loop is a subloop if it is closed under multiplication and inverses.

Let $Q$ be a loop with inverses, and let
\[
    J:Q\to Q,\quad u\mapsto u^{-1}
\]
be the inversion map. Since $u^{-1}\varphi = (u\varphi)^{-1}$ for every $u\in Q$ and $\varphi\in\aut{Q}$, we have $\varphi^J = \varphi$ for every $\varphi\in\aut{Q}$.

A loop $Q$ with inverses has the \emph{automorphic inverse property} if it satisfies the identity
\begin{equation}\label{Eq:AIP}
    (uv)^{-1} = u^{-1}v^{-1},\tag{AIP}
\end{equation}
or, equivalently, if $R_v^J= R_{v^{-1}}$ holds for every $v\in Q$.

A loop is \emph{right Bruck} if it satisfies \eqref{Eq:RBol} and \eqref{Eq:AIP}. Therefore, in a right Bruck loop we have $R_v^J = R_{v^{-1}} = R_v^{-1}$.

Let $Q$ be a loop and $\varphi\in\aut{Q}$. Then we have $(v\varphi^{-1}\cdot u)\varphi = v\cdot u\varphi$ for every $u$, $v\in Q$, and thus $R_u^\varphi = R_{u\varphi}$ for every $u\in Q$.

A loop $Q$ is \emph{right automorphic} if $\rinn{Q}\le\aut{Q}$. It is well known that right Bruck loops are right automorphic.

\subsection{Isotopes and conjugation in right Bol loops}

Two loops $(Q_1,\cdot)$, $(Q_2,*)$ are \emph{isotopic} if there exist bijections $f$, $g$, $h:Q_1\to Q_2$ such that $f(u)*g(v)=h(u\cdot v)$ for every $u$, $v\in Q_1$. We then say that $(Q_2,*)$ is a \emph{loop isotope} of $(Q_1,\cdot)$.

\begin{lemma}\label{Lm:OrderLu}
Let $Q$ be a right Bol loop. Then every loop isotope of $Q$ is isomorphic to a loop isotope of the form $(Q,\circ_c)$, where
\begin{equation}\label{Eq:Isotope}
    u\circ_c v = (u\cdot vc)c^{-1}.
\end{equation}
Moreover, for all $u\in Q$, the order of $u$ in $(Q,\circ_c)$ is the length of the orbit of $L_u$ through $c$.
\end{lemma}
\begin{proof}
By \cite[Lemma 3.4]{Robinson}, every loop isotope of $Q$ is isomorphic to a loop isotope $(Q,\ast_c)$, where
\[
    u\ast_c v = uc\cdot (c\backslash v).
\]
By \eqref{Eq:LeftDivision}, $u\ast_c v = uc\cdot (c^{-1}\cdot vc)c^{-1} = [(uc\cdot c^{-1})\cdot vc]c^{-1} = [u\cdot vc]c^{-1} = u\circ_c v$, where we have also used \eqref{Eq:RBol} and the right inverse property.

By an easy induction, the $n$th power of $u$ in $(Q,\circ_c)$ is $1(L_u^{\circ})^n = (cL_u^n)c^{-1}$, where $L_u^\circ$ denotes the left translation by $u$ in $(Q,\circ_c)$.
\end{proof}

\begin{lemma}\label{Lm:RL}
Let $Q$ be a right Bol loop. For all $u$, $v\in Q$ and all positive integers $m$,
\[
    R_{uv} ( R_v^{-1} R_{uv} )^{m-1} = R_{vL_u^m}.
\]
\end{lemma}
\begin{proof}
The case $m=1$ is trivial. For $m=2$, we compute
\[
    R_{uv} R_v^{-1} R_{uv} = R_{uv}R_{v^{-1}}R_{uv} = R_{(uv\cdot v^{-1})\cdot uv} = R_{u\cdot uv}.
\]
Now suppose that the result holds up to $m-1>1$ and compute
\begin{align*}
    R_{vL_u^m} &= R_{(uv)L_u^{m-1}} = R_{u\cdot uv} ( R_{uv}^{-1} R_{u\cdot uv} )^{m-2} \\
        &= R_{uv} R_v^{-1} R_{uv} ( R_{uv}^{-1} R_{uv} R_v^{-1} R_{uv} )^{m-2} = R_{uv} ( R_v^{-1} R_{uv} )^{m-1},
\end{align*}
using the induction step (with $uv$ in place of $v$) in the second equality and the case $m=2$ in the third equality.
\end{proof}

Let $T_u=R_uL_u^{-1}$.

\begin{proposition}\label{Pr:T}
Let $Q$ be a right Bol loop. Then for all $u$, $v\in Q$ and for all $k\ge 0$,
\[
    (vT_u)^k = (u^{-1}\cdot (u^2)L_v^k)u^{-1}.
\]
\end{proposition}
\begin{proof}
First, using \eqref{Eq:LeftDivision}, we can express $vT_u$ equationally as
\begin{equation}\label{Eq:T}
    vT_u = (u^{-1}\cdot vu^2)u^{-1}.
\end{equation}
The case $k=0$ is clear. For $k>0$, \eqref{Eq:T} yields
\begin{displaymath}
    (vT_u)^k    = vT_u R_{vT_u}^{k-1} = ((u^{-1}\cdot vu^2)u^{-1}) R_{(u^{-1}\cdot vu^2)u^{-1}}^{k-1}.
\end{displaymath}
Using \eqref{Eq:RBol}, the last expression can be rewritten as
\begin{displaymath}
    (u^{-1})R_{vu^2} R_u^{-1} ( R_u^{-1} R_{vu^2} R_u^{-1} )^{k-1} = (u^{-1})R_{vu^2} (R_u^{-1} R_u^{-1} R_{vu^2} )^{k-1} R_u^{-1}.
\end{displaymath}
Using the right power alternative property and Lemma \ref{Lm:RL}, we can further rewrite this as
\begin{displaymath}
    (u^{-1})R_{vu^2} (R_u^{-2} R_{vu^2})^{k-1} R_u^{-1} = (u^{-1})R_{(u^2)L_v^k} R_u^{-1} = (u^{-1}\cdot (u^2)L_v^k)u^{-1},
\end{displaymath}
completing the proof.
\end{proof}

\subsection{Bruck loops and twisted subgroups}

A subset $S$ of a loop $Q$ is \emph{uniquely $2$-divisible} if the mapping $x\mapsto x^2$ is a bijection of $S$. Given $x\in S$, the unique element $y\in S$ such that $y^2=x$ will be denoted by $x^{1/2}$.

The theory of right Bruck loops of odd order (and, more generally, of uniquely $2$-divisible right Bruck loops) has been greatly developed by G.~Glauberman in 1960s. The following result is excerpted from \cite{GlaubermanI}.

\begin{theorem}[\cite{GlaubermanI}]\label{Th:Gl}
Let $Q$ be a right Bruck loop of odd order. Then:
\begin{enumerate}
\item[(i)] For all $u$, $v\in Q$, the identities $(uv)^2=(vu^2)v$ and $R_{uv}=(R_vR_u^2R_v)^{1/2}$ hold.
\item[(ii)] $Q$ is solvable.
\item[(iii)] $|\rmlt{Q}|$ has the same prime factors as $|Q|$.
\item[(iv)] For every prime $p$ dividing $|Q|$, there is an element of order $p$ in $Q$.
\end{enumerate}
\end{theorem}

The conclusion of Theorem \ref{Th:Gl}(iii) remains valid in the more general setting of right Bol loops of odd order:

\begin{proposition}[\cite{FoKiPh, GlaubermanI}]
Let $Q$ be a right Bol loop of odd order. Then $|\rmlt{Q}|$ has the same prime factors as $|Q|$.
\end{proposition}

A subset $T$ of a group $H$ is a \emph{twisted subgroup} of $H$ if it contains the identity element, is closed under inverses, and is closed under the product $(u,v)\mapsto uvu$.

The following result appeared essentially as \cite[Lemma 3]{GlaubermanI}. For a proof in the more modern terminology of twisted subgroups, see \cite[Proposition 2.3]{Vojtechovsky}.

\begin{proposition}[\cite{GlaubermanI, Vojtechovsky}]
Let $T$ be a uniquely $2$-divisible twisted subgroup of a group $H$. Then $(T,\circ)$ with multiplication $u\circ v = (vu^2v)^{1/2}$ is a right Bruck loop. Moreover, the powers and inverses in $(T,\cdot)$ and $(T,\circ)$ coincide.
\end{proposition}

The right Bol identity $R_uR_vR_u = R_{(uv)u}$ together with $R_u^{-1} = R_{u^{-1}}$ show that in a right Bol loop $Q$ the right section $R_Q$ is a twisted subgroup of $\rmlt{Q}$, and hence $(R_Q,\circ)$ with multiplication $R_u\circ R_v = (R_vR_u^2R_v)^{1/2}$ is a right Bruck loop. Using the evaluation map $\varphi\mapsto 1\varphi$, we obtain:

\begin{proposition}[\cite{FoKiPh, Vojtechovsky}]\label{Pr:AssociatedBruck}
Let $(Q,\cdot)$ be a uniquely $2$-divisible right Bol loop. Then $(Q,\circ)$ defined by $u\circ v = (vu^2\cdot v)^{1/2}$ is a right Bruck loop. Moreover, powers and inverses in $(Q,\cdot)$ coincide with those in $(Q,\circ)$.
\end{proposition}

When $(Q,\cdot)$ is a uniquely $2$-divisible right Bol loop, we call the loop $(Q,\circ)$ of Proposition \ref{Pr:AssociatedBruck} the \emph{right Bruck loop associated with} $(Q,\cdot)$.

\begin{lemma}\label{Lm:LxRx}
Let $Q$ be a uniquely $2$-divisible right Bol loop and let $(Q,\circ)$ be the associated right Bruck loop. Then $\rmlt{Q,\circ}$ is conjugate to the group $\genof{L_uR_u}{u\in Q}$ in the symmetric group on $Q$.
\end{lemma}
\begin{proof}
Let $\sigma:Q\to Q$ be the squaring map, and let $R_v^\circ$ be the right translation by $v$ in $(Q,\circ)$. Then for every $u$, $v\in Q$ we have $uR_v^\circ\sigma = (u\circ v)^2 = vu^2\cdot v = u\sigma L_vR_v$, so $\sigma^{-1}R_v^\circ\sigma = L_vR_v$.
\end{proof}

Here is another source of twisted subgroups. Let $H$ be a group and $\tau\in\aut{H}$. Then the set
\[
    K_H(\tau) = \setof{u\in H}{u\tau=u^{-1}}
\]
of anti fixed points of $\tau$ is a twisted subgroup of $H$.

\begin{lemma}[\cite{FoKiPh, Vojtechovsky}]\label{Lm:Twisted}
Let $T$ be a twisted subgroup of a uniquely $2$-divisible group $H$, let $\tau\in\aut{H}$, and suppose that $T\le K_H(\tau)$ and $\langle T\rangle = H$. Then $T=K_H(\tau)$.
\end{lemma}

\subsection{The group $\rmlt{Q}\rtimes\langle J\rangle$ for right Bruck loops of odd order}

Recall that in a right Bruck loop the inversion map $J$ acts on the right section via $R_u^J = R_u^{-1}$. We will investigate the group $A = \rmlt{Q}\rtimes\langle J\rangle$ for a right Bruck loop $Q$ of odd order and show that the isomorphism type of $A$ determines the isomorphism type of $Q$.

\begin{lemma}\label{Lm:Bloop}
Let $Q$ be a right Bruck loop of odd order. Let $G=\rmlt{Q}$ and $A = G\rtimes \langle J\rangle$. Then:
\begin{enumerate}
\item[(i)] All elements of $A\setminus G$ have even order.
\item[(ii)] $\rinn{Q}$ coincides with the fixed points of the action of $J$ on $G$.
\item[(iii)] The right section $R_Q$ coincides with the set $K_G(J)$ of anti fixed points of $J$ in $G$.
\item[(iv)] The mutually inverse bijections $f:S\mapsto R_S=\setof{R_u\in G}{u\in S}$ and $g: T\mapsto 1T =\setof{1\varphi}{\varphi\in T}$ form a one-to-one correspondence between subloops of $Q$ and twisted subgroups of $G$ contained in $R_Q$.
\item[(v)] Let $T$ be a twisted subgroup of $G$ contained in $R_Q$. Then there is a subgroup $U\le G$ such that $T=R_Q\cap U$.
\item[(vi)] Let $T$ be a twisted subgroup of $G$ contained in $R_Q$, and let $1T=S$ be the corresponding subloop of $Q$. Then $S\unlhd Q$ if and only if $T=R_Q\cap U$ for some normal, $J$-invariant subgroup $U$ of $G$.
\end{enumerate}
\end{lemma}

\begin{proof}
(i) This is clear from the fact that $J$ is an involution and from the multiplication in $G \rtimes \langle J\rangle$.

(ii) Every element of $G$ is of the form $\psi R_u$ for some $\psi\in\rinn{Q}\le\aut{Q}$ and $u\in Q$. We have $(\psi R_u)^J = \psi^J R_u^J = \psi R_{u^{-1}}$. Thus $\psi R_u$ is a fixed point of $J$ if and only if $R_u=R_{u^{-1}}$, which happens if and only if $u=u^{-1}$, which is equivalent to $u=1$, since $Q$ has no elements of even order.

(iii) Recall that $R_Q$ is a twisted subgroup of $G=\langle R_Q\rangle$. We have $R_Q\le K_G(J)$ thanks to $R_u^J = R_{u^{-1}}$. By Theorem \ref{Th:Gl}, $G$ is of odd order, hence uniquely $2$-divisible. Lemma \ref{Lm:Twisted} then implies that $R_Q = K_G(J)$.

(iv) For a subloop $S$ of $Q$, we claim that $Sf=R_S$ is a twisted subgroup of $G$. Indeed, if $u$, $v\in S$ then $R_uR_vR_u = R_{(uv)u}\in R_S$ and $R_u^{-1} = R_{u^{-1}}\in R_S$. For a twisted subgroup $T$ of $G$ such that $T\subseteq R_Q$, we claim that $Tg = 1T$ is a subloop of $Q$. Indeed, if $R_u$, $R_v\in T$ then $R_u\circ R_v = (R_vR_u^2R_v)^{1/2} = R_{uv}\in T$ by Theorem \ref{Th:Gl}, so $1R_u\cdot 1R_v = uv = 1R_{uv}\in 1T$, and we also have $(1R_u)^{-1} = u^{-1} = 1R_u^{-1}\in 1T$. Finally, we have $Sfg = (R_S)g = S$ and $Tgf = (1T)f = R_{1T} = T$.

(v) By (iv), $T=R_S$ for some $S\le Q$. Let $U=\langle R_S\rangle\le G$. If $R_v\in U$ then $v=1R_v = 1R_{u_1}\cdots R_{u_k}$ for some $u_i\in S$, so $v\in S$. This shows that $R_Q\cap U=R_S=T$.

(vi) If $T=R_Q\cap U$ for some normal (not necessarily $J$-invariant) subgroup $U$ of $G$, then $S=1T$ is a normal subloop of $Q$ by \cite[Lemma 5(v)]{GlaubermanI}. Conversely, suppose that $1T=S\unlhd Q$. The projection $Q\to Q/S$ induces a surjective homomorphism $\alpha:G=\rmlt{Q} \to \rmlt{Q/S}$. For any $u\in Q$, we have $R_u\alpha=1$ if and only if $u\in S$. Hence $R_Q\cap \ker{\alpha} = R_S = T$. Moreover, $\ker{\alpha}$ is $J$-invariant.
\end{proof}

\begin{proposition}\label{Pr:BruckIsom}
For $i\in\{1,2\}$, let $Q_i$ be a right Bruck loop of odd order with inversion map $J_i$, and let $A_i=\rmlt{Q_i}\rtimes \langle J_i\rangle$. Then $Q_1\cong Q_2$ if and only if $A_1\cong A_2$.
\end{proposition}
\begin{proof}
If $Q_1\cong Q_2$ then clearly $A_1\cong A_2$. For the converse, suppose that $f:A_1\to A_2$ is an isomorphism. Since $\langle J_2 \rangle$ is a Sylow $2$-subgroup in $A_2$ and all Sylow $2$-subgroups of $A_2$ are conjugate, we can compose $f$ with an inner automorphism of $A_2$ and assume without loss of generality that $J_1f=J_2$. By Lemma \ref{Lm:Bloop}(i), the group $G_i=\rmlt{Q_i}$ is the maximal normal subgroup of odd order in $A_i$, hence $G_1f = G_2$. By Lemma \ref{Lm:Bloop}(iii), $R_{Q_i}=K_{G_i}(J_i)$ is the set of anti fixed points of $J_i$. Given $\psi\in R_{Q_1}$, we then have $(\psi f)^{-1} = (\psi^{-1})f = (\psi^{J_1})f = (J_1\psi J_1)f = J_1f\cdot \psi f\cdot J_1f = J_2\cdot \psi f \cdot J_2 = (\psi f)^{J_2}$, so $\psi f\in R_{Q_2}$. It follows that $f$ induces a bijection $R_{Q_1}\to R_{Q_2}$. Define $\alpha:Q_1\to Q_2$ by $R_uf =R_{u\alpha}$. We claim that $\alpha$ is a loop isomorphism. Indeed, $R_{(uv)\alpha} = R_{uv}f = (R_vR_u^2R_v)^{1/2}f =(R_vf(R_uf)^2R_vf)^{1/2} =(R_{v\alpha}R_{u\alpha}^2R_{v\alpha})^{1/2} =R_{u\alpha v\alpha}$.
\end{proof}

We will eventually show that there is only one isomorphism type of $A=\rmlt{Q}\rtimes\langle J\rangle$ for nonassociative right Bruck loops of order $pq$, from which we deduce via Proposition \ref{Pr:BruckIsom} that there is only one nonassociative right Bruck loop of order $pq$.

\section{A first glance at Bol loops of order $pq$}\label{Sc:First}

Let $p>q$ be odd primes and let $Q$ be a right Bol loop of order $pq$. Many results of \cite{NiederreiterRobinson_pq} on $Q$ were obtained under the assumption that the unique subloop of order $p$ is normal in $Q$. Crucially, the normality assumption was used in \cite{NiederreiterRobinson_pq} to prove that $Q=\setof{b^ja^i\cdot b^j}{0\le i<q,\,0\le j<p}$, where $a\in Q$ is an element of order $q$ and $b\in Q$ is an element of order $p$.

In this section we improve upon several results of \cite{NiederreiterRobinson_pq} by removing the normality assumption. Moreover, we then prove in Theorem \ref{Th:Normal_p} that the unique subloop of order $p$ is normal in $Q$. (In Section \ref{Sc:Second} we give another proof of this fact, and we also show that $q$ divides $p^2-1$ when $Q$ is nonassociative.) We purposely give an argument that is independent of \cite{NiederreiterRobinson_pq} because the normality assumption is interwoven into proofs of \cite{NiederreiterRobinson_pq}.

\subsection{Uniqueness of a subloop of order $p$ in right Bol loops of order $pq$}

\begin{lemma}\label{Lm:RightPA}
Let $p>q$ be odd primes and let $Q$ be a right power alternative loop of order $pq$. Then $Q$ has at most one subloop of order $p$.
\end{lemma}
\begin{proof}
Let $u$, $v$ be two elements of order $p$ in $Q$ such that $\langle u\rangle \cap \langle v \rangle = 1$. Let us consider the orbits $O(w)$ of $w$ under $R_u$. We claim that for every $k$ we have $O(v^k)\cap \langle v\rangle = \{v^k\}$. Indeed, if $v^j\in O(v^k)$ then $v^j = v^k R_u^i = v^k R_{u^i} = v^k u^i$ for some $i$, so $u^i = v^k\backslash v^j = v^{j-k} \in \langle u \rangle \cap \langle v\rangle = 1$, so we can assume without loss of generality that $i=0$ and $k=j$. Because $Q$ is right power alternative, all orbits of $R_u$ have length $|u|$. It follows that $R_u$ has at least $p$ distinct orbits of length $p$, so $|Q|\ge p^2>pq$, a contradiction.
\end{proof}

\begin{corollary}[{\cite[Theorem 1]{NiederreiterRobinson_pq}}]\label{Cr:NiRo}
Let $p>q$ be odd primes and let $Q$ be a right Bol loop of order $pq$. Then $Q$ contains a unique subloop of order $p$. When $Q$ is nonassociative, then all nonidentity elements of $Q$ are of order $p$ or $q$.
\end{corollary}
\begin{proof}
By Lemma \ref{Lm:RightPA}, there is at most one subloop of order $p$ in $Q$. By Theorem \ref{Th:Gl}, $Q$ contains an element of order $p$. Hence $Q$ contains a unique subloop $S$ of order $p$.

Suppose that $Q$ is nonassociative. In any right Bol loop the order of an element is a divisor of the order of the loop. If $Q$ contains an element of order $pq$, it is isomorphic to $\Z_{pq}$, a contradiction. Hence all elements of $Q\setminus S$ have order $q$.
\end{proof}

\begin{lemma}\label{Lm:Normal_q}
Let $p>q$ be odd primes and suppose that $Q$ is a right Bol loop containing a normal subloop of order $q$. Then $Q$ is associative.
\end{lemma}
\begin{proof}
Suppose that $Q$ is nonassociative. Let $S$ be a normal subloop of $Q$ of order $q$, say $S=\langle a\rangle$. By Corollary \ref{Cr:NiRo}, there is $b\in Q$ such that $|b|=p$, and every element of $Q\setminus\langle b\rangle$ has order $q$. Thus $1=(ba)^q\in (bS)^q = b^qS$, so $b^q\in S$, a contradiction with $|b|=p$.
\end{proof}

\subsection{Factorizations in right Bol loops of order $pq$}

\begin{lemma}\label{Lm:UniqueSolution}
Let $Q$ be a uniquely $2$-divisible right Bol loop. For all $v\in Q$ and $m,n\in\mathbb Z$ the equation $uv^n\cdot u = v^m$ has a unique solution $u = v^{(m-n)/2}$.
\end{lemma}
\begin{proof}
That $u = v^{(m-n)/2}$ solves the equation is clear thanks to power associativity. For the uniqueness, consider the associated right Bruck loop $(Q,\circ)$ and note that the identity $uv^n\cdot u = v^m$ in $(Q,\cdot)$ is equivalent to the identity $(v^{n/2}\circ u)^2 = v^m$ in $(Q,\circ)$. The latter equation obviously has a unique solution $u$.
\end{proof}

\begin{proposition}[{compare \cite[Theorem 1]{NiederreiterRobinson_pq}}]\label{Pr:Factorizations}
Let $Q$ be a right Bol loop of order $pq$, where $p>q$ are odd primes. Then there exists an element $a\in Q$ of order $q$, an element $b\in Q$ of order $p$, and whenever $a$, $b$ are such elements, we have $Q=\setof{a^ib^j}{0\le i<q,\,0\le j<p} = \setof{b^ja^i}{0\le i<q,\,0\le j<p} = \setof{b^ja^i\cdot b^j}{0\le i<q,\,0\le j<p}$.
\end{proposition}
\begin{proof}
In each case, it is enough to show that no two elements of the given form coincide, for then there are precisely $pq$ elements of that form. If $a^i b^j = a^k b^{\ell}$, then $a^i b^{j-\ell} = a^k$ by the right inverse property and so $b^{j-\ell} = a^i\backslash a^k = a^{k-i}$. Since $\langle a\rangle \cap \langle b\rangle = 1$, we must have $j=\ell$ and $k=i$. The argument for the form $b^ja^i$ is similar.

Suppose that $b^j a^i\cdot b^j = b^k a^{\ell} \cdot b^k$. Then $R_b^j R_a^i R_b^j = R_{b^ja^i\cdot b^j} = R_{b^ka^\ell\cdot b^k} = R_b^k R_a^{\ell} R_b^k$ and so $R_{a^i} = R_b^{k-j} R_a^{\ell} R_b^{k-j} = R_{b^{k-j}a^{\ell}\cdot b^{k-j}}$. Hence $a^i = b^{k-j}a^{\ell}\cdot b^{k-j}$. By Lemma \ref{Lm:UniqueSolution}, we have that $b^{k-j} = a^{(i-\ell)/2}$. As above, we conclude that $k=j$ and $i=\ell$.
\end{proof}

\begin{corollary}\label{Cr:Factorizations}
Let $p>q$ be odd primes and let $Q$ be a right Bol loop of order $pq$. Let $a\in Q$ be of order $q$ and $b\in Q$ of order $p$. Then $\rmlt{Q} = \langle R_a,\,R_b\rangle$.
\end{corollary}
\begin{proof}
By Proposition \ref{Pr:Factorizations}, every $u\in Q$ can be written as $u = b^ja^i\cdot b^j$ for some $i$, $j$. Then $R_u = R_{b^ja^i\cdot b^j} = R_{b^j}R_{a^i}R_{b^j} = R_b^jR_a^iR_b^j \in \langle R_a,\,R_b\rangle$ by \eqref{Eq:RBol} and the right power alternative property.
\end{proof}

\subsection{Complete mappings and right Bol loops of order $pq$}

Throughout this subsection, suppose that $p>q$ are odd primes, and let $Q$ be a right Bol loop of order $pq$. In addition, let $a\in Q$ be an element of order $q$, and $b\in Q$ an element of order $p$.

\begin{lemma}
For each $u\in Q$, the order of $u$ is equal to the order of $L_u$. In particular, $L_a^q = L_b^p = 1$.
\end{lemma}
\begin{proof}
The conclusion is certainly true when $Q$ is associative, so we can assume that $Q$ is not associative. Every loop isotope of $Q$ is then also a nonassociative right Bol loop of order $pq$. Thus for each $c\in Q$, Corollary \ref{Cr:NiRo} implies that each nonidentity element of $Q$ has order $q$ or $p$ in the isotope $(Q,\circ_c)$ defined by \eqref{Eq:Isotope}. By Lemma \ref{Lm:OrderLu} it follows that for each nonidentity element $u\in Q$, the orbits of $L_u$ each have length $q$ or $p$.

Now for $1\ne u\in Q$, $L_u$ has, say, $r$ orbits of length $q$ and $s$ orbits of length $p$. Then $rq + sp = pq$. If $r > 0$ and $s > 0$, then since $sp = (p-r)q$, it must be the case that $q$ divides $sp$ and hence $q$ divides $s$. But $s<q$, a contradiction. Therefore either $r=0$ or $s=0$. Thus $L_u$ has order $q$ or order $p$, and this coincides with the order of $u$ which is the length of the orbit of $L_u$ through $1$.
\end{proof}

\begin{lemma}\label{Lm:Tb}
For all $u\in Q$, $\langle b\rangle T_u \subseteq \langle b\rangle$.
\end{lemma}
\begin{proof}
Since $\langle b\rangle$ is generated by any of its nonidentity elements, it is enough to show $bT_u\in \langle b\rangle$. By Proposition \ref{Eq:T}, we have
\[
    (bT_u)^p = (u^{-1}\cdot (u^2)L_b^p)u^{-1} = u^{-1}u^2\cdot u^{-1} = 1.
\]
Thus $bT_u$ has order dividing $p$. Since $\langle b\rangle$ is the unique subloop of order $p$ by Corollary \ref{Cr:NiRo}, we have $bT_u\in \langle b\rangle$ as claimed.
\end{proof}

Let us now apply Lemma \ref{Lm:Tb} and derive a multiplication formula for $Q$. For every $0\le i<q$ and $0\le j<p$, there exists $0\le j'<p$ such that
\begin{equation}\label{Eq:Theta}
    b^j T_{a^i} = a^i\backslash (b^ja^i) = b^{j'}.
\end{equation}
This gives rise to mappings $\theta_i:\Z_p\to \Z_p$, $j\mapsto j'=j\theta_i$.

A bijection $f$ of a group $(G,+)$ is a \emph{complete mapping} if the mapping $x\mapsto xf+x$ is also a bijection of $G$.

\begin{lemma}\label{Lm:BasicThetaProperties}
For every $0\le i<q$, the mappings $\theta_i$ defined by \eqref{Eq:Theta} are complete mappings of $\Z_p$. Moreover, $\theta_0$ is the identity mapping and $0\theta_i=0$ for every $i\in \Z_q$.
\end{lemma}
\begin{proof}
By Lemma \ref{Lm:Tb}, $\theta_i$ is a permutation of $\Z_p$. Next,
\[
    b^j a^i \cdot b^j = a^i b^{j\theta_i}\cdot b^j = a^i b^{j\theta_i + j}.
\]
By Proposition \ref{Pr:Factorizations}, the left hand side accounts for $p$ distinct elements of $Q$ as $j$ varies over $\mathbb Z_p$, and hence so does the right hand side. Then Proposition \ref{Pr:Factorizations} implies that $j\mapsto j + j\theta_i$ is a permutation of $\Z_p$.

Since $b^j = b^j a^0 = a^0 b^{j\theta_0} = b^{j\theta_0}$, we must have $\theta_0=1$. Similarly, $a^i = b^0a^i = a^i b^{0\theta_i}$ implies $0\theta_i=0$.
\end{proof}

\begin{proposition}\label{Pr:MultiplicationFormula}
The multiplication in $Q$ is uniquely determined by the complete mappings $\theta_i: \Z_p\to \Z_p$ defined by \eqref{Eq:Theta}. In particular, for $0\le i$, $k<q$ and $0\le j$, $\ell<p$, we have
\begin{equation}\label{Eq:Mult}
    a^i b^j \cdot a^k b^{\ell} = a^{i+k} b^{m + (j+m)\theta_i^{-1}\theta_{i+k}},
\end{equation}
where $m + m\theta_k = \ell$.
\end{proposition}
\begin{proof}
Fix $i$, $k\in \Z_q$ and $j$, $\ell\in \Z_p$, and let $m\in \Z_p$ be the unique element satisfying $m + m\theta_k = \ell$. Then
\begin{align*}
    a^i b^j \cdot a^k b^{\ell} &= a^i b^j \cdot a^k b^{m + m\theta_k} = a^i b^j\cdot (a^k b^{m\theta_k}\cdot b^m) = a^i b^j\cdot (b^m a^k \cdot b^m) \\
        &= (a^i b^j\cdot b^m)a^k \cdot b^m = (a^i b^{j+m})a^k\cdot b^m = (b^{(j+m)\theta_i^{-1}} a^i)a^k\cdot b^m \\
        &= b^{(j+m)\theta_i^{-1}} a^{i+k}\cdot b^m = a^{i+k} b^{(j+m)\theta_i^{-1}\theta_{i+k}}\cdot b^m = a^{i+k} b^{(j+m)\theta_i^{-1}\theta_{i+k} + m}.
\end{align*}
\end{proof}

\begin{theorem}\label{Th:Normal_p}
Let $p>q$ be odd primes and let $Q$ be a right Bol loop of order $pq$. Then the unique subloop of order $p$ is normal in $Q$.
\end{theorem}
\begin{proof}
Let $a\in Q$ be an element of order $q$ and $b\in Q$ an element of order $p$. Define $\phi : Q\to \mathbb{Z}_q$ by $(a^ib^j)\phi = a^i$. It follows from Proposition \ref{Pr:Factorizations} that $\phi$ is well-defined, and from Proposition \ref{Pr:MultiplicationFormula} that $\phi$ is a homomorphism with kernel $\langle b\rangle$.
\end{proof}

Call a complete mapping $\theta:\Z_p\to \Z_p$ \emph{linear} if $(i+j)\theta = i\theta + j\theta$ for every $i$, $j\in \Z_p$. Equivalently, a complete mapping $\theta:\Z_p\to \Z_p$ is linear if there is $\lambda\in \Z_p\setminus\{0,-1\}$ such that $i\theta = \lambda\cdot i$ for every $i\in \Z_p$.

We record a useful corollary of Proposition \ref{Pr:MultiplicationFormula}.

\begin{corollary}\label{Cr:LinearAssociative}
If every $\theta_i$ is linear and we have $\theta_{i+k}=\theta_i\theta_k$ for every $i$, $k$, then $Q$ is associative.
\end{corollary}
\begin{proof}
We have
\begin{displaymath}
    a^ib^j\cdot a^kb^\ell = a^{i+k}b^{m+(j+m)\theta_k} = a^{i+k} b^{m+m\theta_k+j\theta_k} = a^{i+k}b^{\ell+j\theta_k}
\end{displaymath}
with the usual convention on $m$, so $Q$ is isomorphic to the semidirect product $(\mathbb Z_q\times\mathbb Z_p,*)$ with associative multiplication
\begin{displaymath}
    (i,j)*(k,\ell) = (i+k,j\theta_k+\ell).
\end{displaymath}
\end{proof}

\section{The Bruck loops of order $pq$}

Let $p>q$ be odd primes. We prove that a nonassociative right Bruck loop of order $pq$ exists if and only if $q$ divides $p^2-1$, and in such case the loop is unique up to isomorphism. Our main tool is Proposition \ref{Pr:BruckIsom}, so we must first obtain some results on groups that can arise as $A=\rmlt{Q}\rtimes \langle J\rangle$.

\subsection{Dihedral groups in $GL(2,p)$}

For a group $N$, denote by $\hol{N}$ the holomorph of $N$, that is, the semidirect product $N\rtimes\aut{N}$ with natural action.

Consider now a semidirect product $G=N\rtimes H$ with conjugation action $\alpha:H\to\aut{N}$. If $\alpha$ is faithful (that is, injective), then $G$ embeds into $\hol{N}$ via the isomorphism $nh\mapsto (n,h\alpha)$.

The goal of this subsection is to show that if $G=N\rtimes H$, $N\cong \Z_p\times \Z_p$, $H\cong D_{2q}$ and $H$ acts faithfully by conjugation on $N$, then $q$ divides $p^2-1$ and $G$ is uniquely determined up to isomorphism. See Proposition \ref{Pr:gr_unique}.

We start by exhibiting a canonical copy of $(\Z_p\times \Z_p)\rtimes D_{2q}$ in $GL(2,p)$. Suppose for a while that $q$ divides $p^2-1$. Note that then $q$ either divides $p-1$ or $p+1$ but not both. Let $\omega$ be a primitive $q$th root of unity in $\F_{p^2}$. Clearly, $\omega\in \F_p$ if and only if $q$ divides $p-1$.

We define $\sigma$, $\tau \in GL(2,p)$ as follows. When $q$ divides $p-1$, let
\[
    \sigma = \begin{pmatrix} \omega&0\\0&\omega^{-1}\end{pmatrix},\quad \tau =\begin{pmatrix} 0&1\\1&0 \end{pmatrix}.
\]
When $q$ divides $p+1$, fix an $\F_p$-basis $B$ of $\F_{p^2}$, and let $\sigma$, $\tau$ be the matrices of the $\F_p$-linear maps $x\mapsto \omega x$, $x\mapsto x^p$ with respect to $B$, respectively. Straightforward calculation shows that in both cases $\sigma$, $\tau$ satisfy the relations $\sigma^q =\tau^2=(\sigma\tau)^2=1$. (When $q$ divides $p+1$, use $\omega^{p+1}=1$.)

Denote by $\Delta$ the subgroup $\langle \sigma, \tau \rangle$ of $GL(2,p)$ and note that $\Delta$ is isomorphic to the dihedral group $D_{2q}$.

\begin{proposition}\label{Pr:GL2p}
When $q$ does not divide $p^2-1$ then $GL(2,p)$ has no subgroup isomorphic to the dihedral group $D_{2q}$. When $q$ divides $p^2-1$ then any subgroup isomorphic to $D_{2q}$ is conjugate to $\Delta$ in $GL(2,p)$.
\end{proposition}
\begin{proof}
Let $G=GL(2,p)$ and note that $|G| = (p^2-1)(p^2-p) = (p-1)^2p(p+1)$. When $q$ does not divide $p^2-1$ then $q$ does not divide $|G|$, so $G$ certainly does not contain a subgroup of order $2q$. For the rest of the proof suppose that $q$ divides $p^2-1$, and let $\sigma$, $\tau$, $\Delta$ be as above. Let $S=\langle s,\,t\;:\; s^q=t^2=(st)^2=1 \rangle$ be a subgroup of $G=GL(2,p)$ isomorphic to $D_{2q}$.

Assume first that $q$ divides $p-1$. Define the subgroup
\[
    U=\left\{ \begin{pmatrix} a&0\\0&b \end{pmatrix}\;:\;a,\,b \in \mathbb{F}_p^* \right\}\cong \Z_{p-1}\times \Z_{p-1}
\]
of $G$. As $[G:U]=p(p+1)$ is coprime to $q$, any $q$-Sylow subgroup of $U$ is a $q$-Sylow subgroup of $G$. In particular, $s$ is conjugate to an element of $U$. We can therefore assume that, up to conjugacy, $s \in U$ and
\[
    s=\begin{pmatrix} \omega^i&0\\0&\omega^{j} \end{pmatrix}
\]
for the primitive $q$th root of unity $\omega$ in $\F_p$ fixed above and for some integers $i,j$. Since $s$ is not central in $G$, $i\neq j$, and $s$ has precisely two $1$-dimensional invariant subspaces, namely $V_1 = \langle (1,0)\rangle$ and $V_2=\langle (0,1)\rangle$. The relation $sts=t$ implies $V_i t=V_ists=V_its$, and thus $V_it\in\{V_1,V_2\}$. If $t$ fixes $V_1$, $V_2$ then $S$ is contained in the abelian group $\F_p^*\times \F_p^*$, a contradiction. Therefore, $t$ must interchange $V_1$ and $V_2$. In particular, $t$ has the form
\[
    t=\begin{pmatrix} 0&a^{-1}\\a&0 \end{pmatrix}
\]
for some $a\in\F_p^*$. The relation $tst=s^{-1}$ is then equivalent to $i=-j$. With
\[
    u=\begin{pmatrix} a&0\\0&1 \end{pmatrix},
\]
we have $s^u=s$ and
\[
    t^u=\begin{pmatrix} 0&1\\1&0 \end{pmatrix}.
\]
Hence $S$ is conjugate to $\Delta$.

Now assume that $q$ divides $p+1$. By \cite[II.7.3.(a)]{Huppert}, $G$ contains a cyclic subgroup $U$ of order $p^2-1$. In fact, $U$ consists of the $\mathbb{F}_p$-linear maps $\rho_a: x\mapsto xa$ with $a\in \mathbb{F}_{p^2}^*$. Furthermore, $C_G(U)=U$ and $N_G(U)/U$ has order $2$, which means $N_G(U)=\{\rho_a, \rho_a\tau \mid a\in \mathbb{F}_{p^2}^*\}$. We claim that any involution of $N_G(U)\setminus U$ is $U$-conjugate to $\tau$. On the one hand, $C_U(\tau)=\{\rho_a \mid a\in \mathbb{F}_p^*\}$, hence $|\tau^U|=(p^2-1)/(p-1)=p+1$. On the other hand, $(\tau \rho_a)^2 =\rho_a^{p+1}$ is the identity map if and only if $a^{p+1}=1$, which means that $N_G(U)\setminus U$ has $p+1$ involutions.

Since $[G:U]$ is coprime to $q$, the $q$-Sylow subgroup of $U$ is a $q$-Sylow subgroup of $G$. Up to conjugacy, we can assume that $\langle s \rangle \leq U$ is the unique subgroup of order $q$ in $U$. Then again by \cite[II.7.3.(a)]{Huppert}, $t\in N_G(\langle s \rangle)\leq N_G(U)$. By our claim above, $t$ and $\tau$ are $U$-conjugates. As $U$ centralizes $s$, the subgroups $\langle s,t \rangle$ and $\Delta$ are conjugate.
\end{proof}

We will identify $\hol{\mathbb F_p^n}$ with the affine linear group $AGL(n,p)$ consisting of all affine linear maps $x\mapsto b+xA$, where $b\in\mathbb F_p^n$ and $A\in\aut{\mathbb F_p^n} = GL(n,p)$. Recall that the socle of $AGL(n,p)$ is in fact the unique minimal normal subgroup of $AGL(n,p)$, namely the group consisting of all translations $x\mapsto b+x$.

\begin{proposition}\label{Pr:gr_unique}
Let $G=NH$ be a group such that $\Z_p\times \Z_p\cong N\unlhd G$, $D_{2q}\cong H\leq G$ and $C_G(N)=N$. Then $q$ divides $p^2-1$ and $G\cong \mathbb{F}_p^2 \rtimes \Delta$.
\end{proposition}
\begin{proof}
Since the orders of $N$ and $H$ are coprime, we have $N\cap H=1$. It follows that $G\cong N\rtimes H$, where the action of $H$ on $N$ is by conjugation. The assumption $C_G(N)=N$ means that the action is faithful, and we have an embedding $\varphi:G\to\hol{N}=AGL(2,p)$. The image $N\varphi$ is the socle of $AGL(2,p)$, and $H\varphi \le GL(2,p)$. By Proposition \ref{Pr:GL2p}, $q$ divides $p^2-1$ and there is an element $g\in GL(2,p)$ with $(H\varphi)^g=\Delta$. Since $(N\varphi)^g = N\varphi$, the map $x\mapsto (x\varphi)^g$ is an isomorphism $G\to \mathbb{F}_p^2 \rtimes \Delta$.
\end{proof}

\subsection{Uniqueness}

Throughout this section, let $Q$ be a right Bruck loop of order $pq$. We prove that either $Q$ is the cyclic group $\Z_{pq}$, or $q$ divides $p^2-1$ and $Q$ is the nonassociative Bruck loop $B_{p,q}$ constructed by Niederreiter and Robinson.

We start with a special case of Theorem \ref{Th:Normal_p}, giving a proof independent of most of the results in Section \ref{Sc:First}.

\begin{proposition}\label{Pr:p}
$Q$ possesses a unique subloop of order $p$ and this subloop is normal.
\end{proposition}
\begin{proof}
By Theorem \ref{Th:Gl}, $Q$ is solvable and so its derived subloop $Q'$ is properly contained in $Q$. If $Q'=1$ then $Q$ is an abelian group of order $pq$ and the result follows. We can therefore assume that $|Q'|\in\{p,q\}$.

If $|Q'|=p$, we are done, so assume that $|Q'|=q$ and $Q/Q'\cong \Z_p$. Let $S$ be the unique subloop of order $p$, whose existence is guaranteed by Corollary \ref{Cr:NiRo}. Recall that $|u|$ divides $|Q|$ in any right Bol loop. Consider any $u\in Q\setminus(Q'\cup S)$. Since $|uQ'|$ divides $|u|$ and $|uQ'|=p$, it follows that $p$ divides $|u|\ne p$. Thus $|u|=pq$, $Q\cong \Z_{pq}$, and $Q'=1$, a contradiction.
\end{proof}

Let $G=\rmlt{Q}$. By Proposition \ref{Pr:Factorizations} and Corollary \ref{Cr:Factorizations}, there are $a$, $b\in Q$ such that $|a|=q$, $|b|=p$ and $G=\langle R_a,\,R_b\rangle$.

As in Lemma \ref{Lm:Bloop}, let $A = G\rtimes \langle J\rangle$. Since $G\unlhd A$, $G/A\cong \Z_2$ and $G$ is solvable of order $p^aq^b$ by Theorem \ref{Th:Gl}, $A$ is solvable of order $2p^aq^b$. Let $N$ be a minimal normal subgroup of $A$, necessarily an elementary abelian $r$-group for some $r\in\{2,\,p,\,q\}$.

\begin{lemma}\label{Lm:r_not_2}
$r\neq 2$ and $N\unlhd G$.
\end{lemma}
\begin{proof}
Suppose first that $r=2$. Then $|N|=2$ because $4$ does not divide $|A|$, and from $N\unlhd A$ we deduce $N\le Z(A)$. Let $\varphi$ be the unique involution of $N$. If $\varphi\ne J$, then $\varphi J = J\varphi$ shows that $\langle \varphi,\,J\rangle$ is a subgroup of order $4$, a contradiction. Thus $N=\{1,J\}$. But then $R_u = R_u^J = R_{u^{-1}}$ implies $|u|\le 2$ for every $u\in Q$, a contradiction. Hence $r\ne 2$. Since $N$ contains no elements of even order, it is a subgroup of $G$ by Lemma \ref{Lm:Bloop}(i).
\end{proof}

Consider a vector space $V$ over a field of odd characteristic, and let $\varphi$ be an involutory automorphism of $V$. Then any $v\in V$ can be written as $v = v^+ + v^-$, where $v^+ = (v+v\varphi)/2$ and $v^- = (v-v\varphi)/2$. Moreover, $v^+\varphi = v^+$ and $v^-\varphi = - v^-$.

Denote by $N^+$ (resp. $N^-$) the fixed points (resp. anti fixed points) of the involutory automorphism $J$ on $N\cong \mathbb F_r^m$. As above, any $u \in N$ can be written as $u = u^+ + u^-$, where $u^+ =(uu^J)^\frac{1}{2}\in N^+$ and $u^-=(u(u^{-1})^J)^\frac{1}{2}\in N^-$. Thus $N=N^+N^-=N^+\times N^-$.

\begin{lemma}\label{Lm:orderN-}
$N^-$ is a twisted subgroup of $G$ contained in $R_Q$, it corresponds to a normal subloop of $Q$, and it has size $r$.
\end{lemma}
\begin{proof}
We have $N^-\le G$ by Lemma \ref{Lm:r_not_2}. By Lemma \ref{Lm:Bloop}(iii), $N^-$ is a twisted subgroup of $G$ contained in $R_Q$. By Lemma \ref{Lm:Bloop}(vi), $N^-$ corresponds to a normal subloop of $Q$. Since $|Q|=pq$ and $|N^-|$ is a power of $r$, it follows that $|N^-|\in\{1,\,r\}$. Suppose that $|N^-|=1$. Then $N=N^+\le\rinn{Q}$ by Lemma \ref{Lm:Bloop}(ii). Together with $N\unlhd G$, this is a contradiction with the fact that $\rinn{Q}$ is a core-free subgroup of $\rmlt{Q}$.
\end{proof}

\begin{lemma}\label{Lm:N}
If $Q$ is not associative then $N^-=\langle R_b\rangle$, $|N^+|=p$ and $N\cong \Z_p\times \Z_p$.
\end{lemma}
\begin{proof}
By Proposition \ref{Pr:p}, $Q$ contains a normal subloop of order $p$. It does not contain a normal subloop of order $q$, otherwise $Q\cong \Z_{pq}$ by Lemma \ref{Lm:Normal_q}, a contradiction. By Lemma \ref{Lm:orderN-}, $|N^-|=p$. Since $N^-$ is a twisted subgroup of $G$ contained in $R_Q$, we must have $N^-=\langle R_b\rangle$ because $|R_u|=|u|$ for every $u\in Q$.

Notice that $G = \langle R_a,\,R_b\rangle = \langle N, R_a \rangle = N\langle R_a \rangle$ and $|R_a|=q$. If $N^+$ is trivial, then $|G|\le pq$, which implies that $Q$ is associative, a contradiction. We can therefore assume that $N^+$ is nontrivial. Let $\dim(N)=k$ as a vector space over $\mathbb F_p$. From $N=N^+\times N^-$ and $\dim(N^-)=1$ we deduce $\dim(N^+)=k-1$, $k\ge 2$.

Consider $J_0=J^{R_a}$. Then $JJ_0=JR_a^{-1}JR_a=R_a^2$. Since $|R_a|=q$ is odd, we have $A=\langle N,\,R_a,\,J\rangle = \langle N,\,R_a^2,\,J\rangle = \langle N,\,J,\,J_0\rangle$. The set of fixed points of $J_0$ in $N$ is $(N^+)^{R_a}$ because $(u^{R_a})^{J_0}=(u^J)^{R_a}=u^{R_a}$ holds if and only if $u\in N^+$. Let $M=N^+\cap(N^+)^{R_a}$ and note that all elements of $M$ are centralized by $N$, $J$ and $J_0$. Hence $M\leq Z(A)$ and, in particular, $M\unlhd A$.

Now, $\dim(M) + k = \dim(M) + \dim(N) \ge \dim(N^+) + \dim((N^+)^{R_a}) = 2(k-1)$, so $\dim(M)\ge k-2$. If $k>2$ then $M$ is a nontrivial normal subgroup of $A$ properly contained in $N$, a contradiction. Thus $k=2$, $\dim(N^+)=1$ and $N\cong \Z_p\times \Z_p$.
\end{proof}

\begin{proposition}\label{Pr:BruckRmltJ}
Let $p>q$ be odd primes, let $Q$ be a nonassociative Bruck loop of order $pq$, $G=\rmlt{Q}$ and $A=G\rtimes \langle J\rangle$. Then:
\begin{enumerate}
\item[(i)] $A=NH$ where $N\cong \Z_p\times \Z_p$, $H\cong D_{2q}$, and $C_A(N)=N$.
\item[(ii)] $q$ divides $p^2-1$.
\item[(iii)] $G$ is isomorphic to $(\Z_p\times \Z_p)\rtimes \Z_q$.
\end{enumerate}
\end{proposition}
\begin{proof}
(i) Lemmas \ref{Lm:r_not_2}--\ref{Lm:N} imply that $A$ has a normal subgroup $N\cong \Z_p\times \Z_p$ with $R_b\in N$. The subgroup $H=\langle R_a,\,J \rangle$ is isomorphic to $D_{2q}$. $A=NH$ holds. Since $C_A(N)$ is normal in $A$, $H_0=H\cap C_A(N)$ is normal in $H$. This implies that if $H_0$ is not trivial then it contains the unique cyclic subgroup of order $q$ of $H$. In particular, $R_a\in H_0$. From $R_a\in C_A(N)$, $R_b\in N$ and $\rmlt{Q}=\langle R_a,\,R_b \rangle$ we deduce that $\rmlt{Q}$ is commutative, a contradiction with nonassociativity of $Q$. Hence, $H_0=\{1\}$ and $C_A(N)=N$ hold.

(ii) By Proposition \ref{Pr:gr_unique}, $q$ divides $p^2-1$.

(iii) Let $K=H\cap G$. By Lemma \ref{Lm:Bloop}(i), $K\cong \Z_q$ and both $N$ and $K$ are subgroups of $G$.
\end{proof}

Recall that whenever $q$ divides $p^2-1$, Niederreiter and Robinson constructed a nonassociative right Bruck loop $B_{p,q}$ of order $pq$.

\begin{theorem}\label{Th:Uniqueness}
Let $p>q$ be odd primes and let $Q$ be a right Bruck loop of order $pq$. Then either $Q\cong \Z_{pq}$, or $q$ divides $p^2-1$ and $Q\cong B_{p,q}.$
\end{theorem}
\begin{proof}
If $Q$ is associative then $Q\cong \Z_{pq}$, so we can assume that $Q$ is not associative. Let us relabel $Q$ as $Q_1$, let $J_1$ be the inversion map in $Q_1$, and let $A_1 = \rmlt{Q_1}\rtimes \langle J_1\rangle$. By Propositions \ref{Pr:gr_unique} and \ref{Pr:BruckRmltJ}, $q$ divides $p^2-1$ and $A_1\cong \mathbb{F}_p^2\rtimes \Delta$.

Let $Q_2=B_{p,q}$, let $J_2$ be the inversion map in $Q_2$, and let $A_2=\rmlt{Q_2}\rtimes\langle J_2\rangle$. Reasoning as before, we get $A_2\cong \mathbb{F}_p^2\rtimes \Delta$, too. In particular, $A_1\cong A_2$, and Proposition \ref{Pr:BruckIsom} implies $Q_1\cong Q_2$.
\end{proof}

We will prove later that in every nonassociative right Bol loop of order $pq$ we have $\lnuc{Q}\cong \Z_p$. Here is a special case for right Bruck loops.

\begin{proposition}
Let $p>q$ be odd primes and let $Q$ be the nonassociative right Bruck loop of order $pq$. Then $\lnuc{Q}\cong \Z_p$.
\end{proposition}
\begin{proof}
We first mimic the beginning of the proof of \cite[Proposition 3.18]{Vojtechovsky}. Let $S$ be the unique normal subgroup of order $p$ in $Q$. Consider the mapping $f:\rinn{Q}\to\aut{S}$, $\varphi f = \varphi|_S$, clearly a homomorphism. The kernel of $f$ is equal to $C=\setof{\varphi\in\rinn{Q}}{\varphi|_S=\mathrm{id}_S}$. Since $\aut{S}\cong\aut{\Z_p}$, we see that $\rinn{Q}/C\le\aut{S}$ is a cyclic group of order dividing $p-1$. However, we know that $|\rinn{Q}|=p$ from Proposition \ref{Pr:BruckRmltJ}(iii), so $|\rinn{Q}/C|=1$ and $|C|=p$. It follows that $S\le \lnuc{Q}$.
\end{proof}

\section{Second glance at right Bol loops of order $pq$}\label{Sc:Second}

\subsection{The divisibility conditions $q\mid p^2-1$}

We will now show that a nonassociative right Bol loop $Q$ of order $pq$ exists for odd primes $p>q$ if and only if $q$ divides $p^2-1$. If we knew that the associated right Bruck loop $(Q,\circ)$ is nonassociative, we would be done by Theorem \ref{Th:Uniqueness}. But it is conceivable that $(Q,\circ)$ is a group and this situation must be carefully excluded. The setup we develop here will be useful later, too. Moreover, it will allow us to give another proof of Theorem \ref{Th:Normal_p} independent of most of Section \ref{Sc:First}.

Let $Q$ be a loop. Recall that a triple of bijections $(\alpha,\beta,\gamma)$ of $Q$ is an \emph{autotopism} of $Q$ if $u\alpha\cdot v\beta = (uv)\gamma$ holds for every $u$, $v\in Q$. The autotopisms of $Q$ form a group under componentwise composition, the \emph{autotopism group} $\atp{Q}$.

We claim that a loop $Q$ is right Bol if and only if $(R_u^{-1},L_uR_u,R_u)\in\atp{Q}$ for every $u\in Q$. Indeed, the condition $(R_u^{-1},L_uR_u,R_u)\in\atp{Q}$ is equivalent to the identity $(w/u)(uv\cdot u) = (wv)u$, which is equivalent to \eqref{Eq:RBol} upon substituting $wu$ for $w$.

\begin{lemma}[\cite{Bruck}]\label{Lm:PrincipalAtp}
Let $Q$ be a loop and let $\alpha$, $\beta$, $\gamma$ be bijections of $Q$. Then:
\begin{enumerate}
\item[(i)] $(\mathrm{id},\beta,\gamma)\in\atp{Q}$ if and only if $\beta=\gamma=R_w$ and $w\in\rnuc{Q}$.
\item[(ii)] $(\alpha,\mathrm{id},\gamma)\in\atp{Q}$ if and only if $\alpha=\gamma=L_w$ and $w\in \lnuc{Q}$.
\end{enumerate}
\end{lemma}

Denote by $\pi_i$ the projection on the $i$th coordinate.

\begin{lemma}\label{Lm:Atp}
Let $Q$ be a uniquely $2$-divisible right Bol loop, let $N=\genof{(R_u^{-1},L_uR_u,R_u)}{u\in Q}\le\atp{Q,\cdot}$, and let $(Q,\circ)$ be the right Bruck loop associated with $(Q,\cdot)$. Then:
\begin{enumerate}
\item[(i)] The projection $\pi_1:N\to\rmlt{Q,\cdot}$ is a surjective homomorphism and there is a subloop $W_1$ of $\rnuc{Q,\cdot}$ such that $\ker{\pi_1}= \setof{(\mathrm{id},R_w,R_w)}{w\in\rnuc{Q,\cdot}}$ is isomorphic to $W_1$.
\item[(ii)] The projection $\pi_2:N\to\genof{L_uR_u}{u\in Q}\cong \rmlt{Q,\circ}$ is a surjective homomorphism, there is a subloop $W_2$ of $\lnuc{Q,\cdot}$ such that $\ker{\pi_2}=\setof{(L_w,\mathrm{id},L_w)}{w\in W_2}$ is anti-isomorphic to $W_2$, and $\ker{\pi_2}\le Z(N)$.
\end{enumerate}
\end{lemma}
\begin{proof}
(i) It is clear that $\pi_1$ is onto $\rmlt{Q,\cdot}$. Let $\varphi_w = (\mathrm{id},R_w,R_w)$. By Lemma \ref{Lm:PrincipalAtp}(i), there is a subset $W_1$ of $\rnuc{Q,\cdot}$ such that $\ker{\pi_1}=\setof{\varphi_w}{w\in W_1}$. We claim that $W_1$ is a subloop of $(Q,\cdot)$ and $\ker{\pi_1}\cong W_1$. If $u$, $v\in W_1$ then $\varphi_u$, $\varphi_v\in\ker{\pi_1}$ and $\varphi_u\varphi_v$, $\varphi_u^{-1}\in\ker{\pi_1}$. We have $\varphi_u^{-1}=\varphi_{u^{-1}}$ thanks to the right inverse property, and $\varphi_u\varphi_v =\varphi_{uv}$ because $v\in\rnuc{Q}$. Therefore $uv$, $u^{-1}\in W_1$.

(ii) It is clear that $\pi_2$ is onto $\genof{L_uR_u}{u\in Q}$, which is isomorphic to $\rmlt{Q,\circ}$ by Lemma \ref{Lm:LxRx}. Let $\psi_w = (L_w,\mathrm{id},L_w)$. By Lemma \ref{Lm:PrincipalAtp}(ii), there is a subset $W_2$ of $\lnuc{Q,\cdot}$ such that $\ker{\pi_2}=\setof{\psi_w}{w\in W_2}$. We claim that $W_2$ is a subloop of $(Q,\cdot)$ and $\ker{\pi_2}$ is anti-isomorphic to $W_2$. If $u$, $v\in W_2$ then $\psi_u$, $\psi_v\in\ker{\pi_2}$ and $\psi_u\psi_v$, $\psi_u^{-1}\in\ker{\pi_2}$. We have $\psi_u\psi_v = \psi_{vu}$ since $v\in\lnuc{Q,\cdot}$, and $\psi_u^{-1}=\psi_{u^{-1}}$ thanks to $u\in\lnuc{Q,\cdot}$. Therefore $vu$, $u^{-1}\in W_1$.

Finally, the condition $w\in\lnuc{Q,\cdot}$ is equivalent to $L_wR_u=R_uL_w$ for every $u\in Q$. It follows that $\psi_w$ commutes with every $(R_u^{-1},L_uR_u,R_u)$, hence $\ker{\pi_2}\le Z(N)$.
\end{proof}

\begin{proposition}\label{Pr:HalfGabor}
Let $(Q,\cdot)$ be a Bol loop of odd order. If the associated right Bruck loop $(Q,\circ)$ is associative then $\rmlt{Q,\cdot}$ is nilpotent of class at most $2$.
\end{proposition}
\begin{proof}
Let $N$, $\pi_1$ and $\pi_2$ be as in Lemma \ref{Lm:Atp}. Then $\ker{\pi_2}\le Z(N)$, and $N/\ker{\pi_2}\cong \rmlt{Q,\circ}\cong (Q,\circ)$ is an abelian group by assumption. Therefore $N$ is nilpotent of class at most $2$. Since $\pi_1(N)=\rmlt{Q,\cdot}$, we are done.
\end{proof}

Recall that in a finite nilpotent group the Sylow subgroups are normal, and the maximal subgroups are normal of prime index.

\begin{lemma}\label{Lm:nilpot}
Let $G$ be a finite nilpotent group and $H$ a core-free subgroup of $G$. Then for any prime divisor $r$ of $|G|$, there is a normal subgroup $M\unlhd G$ such that $H\leq M$ and $r=|G/M|$.
\end{lemma}
\begin{proof}
We use induction on $|G|$. Let $N_1$ be a maximal subgroup of $G$ containing $H$. Then $N_1$ is normal in $G$ of prime index $p$. If $r=p$ then we are done. Assume $r\neq p$. By induction hypothesis, $N_1$ has a normal subgroup $N_2$ with $H\leq N_2$ and $r=|N_1/N_2|$. Let $P$ be a Sylow $p$-subgroup of $G$ and define $M=N_2P$. Since $H$ is core-free and $P\unlhd G$, we see that $P$ is not a subgroup of $N_2$ and $|M/N_2|=p$. This implies that $|G/M|=r$, $M$ is maximal in $G$ and hence $M\unlhd G$.
\end{proof}

\begin{lemma}[{{\cite[2.12]{Aschbacher2006}}}]\label{Lm:Aschbacher}
Let $Q$ be a loop and $M$ a group with $\rinn{Q}\leq M \unlhd \rmlt{Q}$. Then $S=\setof{x\in Q}{R_x\in M}=1M$ is a normal subloop of $Q$. Moreover, $Q/S \cong \rmlt{Q}/M$ holds.
\end{lemma}

\begin{proposition}\label{Pr:Index}
Let $Q$ be a finite loop such that $\rmlt{Q}$ is nilpotent. Then for any prime $r$ dividing $|Q|$ there is a normal subloop of $Q$ with index $r$.
\end{proposition}
\begin{proof}
Put $G=\rmlt{Q}$ and $H=\rinn{Q}$. Since $H$ is core-free in $G$, there is a normal subgroup $M$ of $G$ of index $r$ containing $H$ by Lemma \ref{Lm:nilpot}. We are done by Lemma \ref{Lm:Aschbacher}.
\end{proof}

\begin{theorem}\label{Th:BolSubloop}
Let $p>q$ be odd primes and let $Q$ be a nonassociative right Bol loop of order $pq$. Then $q$ divides $p^2-1$.
\end{theorem}
\begin{proof}
Let $(Q,\circ)$ be the associated right Bruck loop. Suppose first that $(Q,\circ)$ is associative. Then $(Q,\circ)$ is in fact an abelian group due to the automorphic inverse property, and $G=\rmlt{Q,\cdot}$ is nilpotent by Proposition \ref{Pr:HalfGabor}. Proposition \ref{Pr:Index} yields a normal subgroup $S$ of index $p$ (hence of order $q$) in $(Q,\cdot)$, a contradiction with Lemma \ref{Lm:Normal_q}. Therefore $(Q,\circ)$ is nonassociative, and we are done by Theorem \ref{Th:Uniqueness}.
\end{proof}

\begin{lemma}\label{Lm:NSylow}
Let $p>q$ be odd primes and let $Q$ be a nonassociative right Bol loop of order $pq$. Let $N$ be as in Lemma \ref{Lm:Atp}. Then $|N|=p^iq^j$ with $i\le 3$, $j\le 2$, and $N$ contains a unique Sylow $p$-subgroup.
\end{lemma}
\begin{proof}
Let $(Q,\circ)$ be the associated right Bruck loop, and let $K_2=\ker{\pi_2}$. By Lemma \ref{Lm:Atp}, $N/K_2\cong\rmlt{Q,\circ}$ and $K_2$ corresponds (in cardinality) to a subloop of $\lnuc{Q}$. If $(Q,\circ)$ is associative then $|\rmlt{Q,\circ}|=pq$, otherwise $|\rmlt{Q,\circ}|=p^2q$ by Proposition \ref{Pr:BruckRmltJ}. Thus $|N|=p^iq^j$ with $i\le 3$, $j\le 2$. Since $p$ does not divide $q^2-1 = (q-1)(q+1)$, $N$ contains a unique Sylow $p$-subgroup.
\end{proof}

\begin{remark}
Here is an alternative proof of Theorem \ref{Th:Normal_p} that does not require most of Section \ref{Sc:First}. We can assume that $Q$ is a nonassociative right Bol loop. Let us resume the argument of Lemma \ref{Lm:NSylow}. Let $P$ be the unique Sylow $p$-subgroup of $N$. We have $N'\le P$ because $|N/P|\in\{q$, $q^2\}$.

Let $K_1=\ker{\pi_1}$. If $K_1\ne 1$ then $\rnuc{Q,\cdot}$ is nontrivial by Lemma \ref{Lm:Atp}, necessarily a normal subloop of order $p$ by Lemma \ref{Lm:Normal_q}. We can therefore suppose that $K_1=1$. Let $G=\rmlt{Q,\cdot}$. Lemma \ref{Lm:Atp} yields $N\cong G$, so $N'\cong G'$ is a $p$-group. Consider the normal subgroup $M=\rinn{Q,\cdot}G'$ of $G$. Since $\rinn{Q,\cdot}$ is core-free in $G$, it does not contain $G'$ and $|M|=p|\rinn{Q,\cdot}|=|G|/q$. By Lemma \ref{Lm:Aschbacher}, $S=1M$ is a normal subloop of $Q$ such that $|Q/S|= |G/M|=q$, so $|S|=p$.
\end{remark}

\subsection{Triviality of the right and middle nuclei}

Recall that in every right inverse property loop $Q$ the right nucleus coincides with the middle nucleus. In this subsection we prove that in a nonassociative right Bol loop of order $pq$ the right and middle nuclei are trivial.

\begin{lemma}\label{Lm:Drapal}
Let $Q$ be a loop, let $S\leq \rnuc{Q}\cap \mnuc{Q}$, and suppose that for some $u\in Q$ we have $S T_u = S$. Then $T_u|_S$ is an automorphism of $S$.
\end{lemma}
\begin{proof}
For $c_1$, $c_2\in S$, we have $u(c_1T_u \cdot c_2T_u) = (u\cdot c_1T_u)\cdot c_2T_u  = c_1 u\cdot c_2T_u = c_1 (u\cdot c_2T_u) = c_1\cdot c_2 u = c_1 c_2\cdot u$. Thus $c_1T_u \cdot c_2T_u = (c_1 c_2)T_u$, as claimed.
\end{proof}

\begin{lemma}
Let $Q$ be a right Bol loop. Then:
\begin{enumerate}
    \item[(i)] For each $u\in Q$, $T_u|_{\rnuc{Q}}$ is an automorphism of $\rnuc{Q}$.
    \item[(ii)] For each $u\in \rnuc{Q}$, $v\in Q$ and $n\ge 0$ we have $uT_{v^n} = uT_v^n$.
\end{enumerate}
\end{lemma}
\begin{proof}
Part (i) follows immediately from Lemma \ref{Lm:Drapal}. Let us prove (ii) by induction on $n$, the case $n=0$ being obvious. Using the fact that $u\in\rnuc{Q}=\mnuc{Q}$, part (i), the inductive step and the right power alternative property, we see that $v^{n+1} \cdot uT_v^{n+1} = v^n(v\cdot uT_v^{n+1}) = v^n(uT_v^n\cdot v) = (v^n\cdot uT_v^n) v = u v^n\cdot v = u v^{n+1}$.
\end{proof}

\begin{theorem}\label{Th:RightNucleus}
Let $p>q$ be odd primes and let $Q$ be a nonassociative right Bol loop of order $pq$. Then $\rnuc{Q}=\mnuc{Q}=1$.
\end{theorem}
\begin{proof}
We know that $\rnuc{Q}$ is a normal subloop of $Q$. If $|\rnuc{Q}|=1$, we are done. If $|\rnuc{Q}|=q$ then $Q$ is associative by Lemma \ref{Lm:Normal_q}, a contradiction. We can therefore assume that $|\rnuc{Q}|=p$, and we must have $\rnuc{Q}=\langle b\rangle$, using the notation of Section \ref{Sc:First}. By Lemma \ref{Lm:Drapal}, we get
\[
    b^{j\theta_i} = b^jT_{a^i} = (bT_{a^i})^j = (bT_a^i)^j = (b^{1\theta_1^i})^j = b^{j\cdot 1\theta_1^i}.
\]
Thus $j\theta_i = j\cdot 1\theta_1^i$ for every $i$, $j$.

Now, if we set $t = 1\theta_1$, then we claim that $1\theta_1^i = t^i$ for all $i\ge 0$. This is clear for $i=0$, so assuming it for $i\ge 0$, we have $1\theta_1^{i+1} = 1\theta_1^i \theta_1 = t^i\theta_1 = t^i\cdot 1\theta_1 = t^{i+1}$.

Summarizing, we have $j\theta_i = t^i j$. By Corollary \ref{Cr:LinearAssociative}, $Q$ is associative, a contradiction.
\end{proof}

\subsection{The left nucleus}

Let $p>q$ be odd primes and let $Q$ be a nonassociative right Bol loop of order $pq$. In this subsection we prove that $|\rmlt{Q}|\in\{p^2q$, $p^3q\}$ and that $\lnuc{Q}$ is a normal subloop of $Q$ isomorphic to $\Z_p$.

\begin{proposition}\label{Pr:RMltBol}
Let $p>q$ be odd primes and let $Q$ be a nonassociative right Bol loop of order $pq$. Let $G=\rmlt{Q}$. Then:
\begin{enumerate}
\item[(i)] $G$ contains a unique Sylow $p$-subgroup.
\item[(ii)] $|G|=p^kq$, where $k\in\{2,\,3\}$.
\end{enumerate}
\end{proposition}
\begin{proof}
(i) Let $N$, $\pi_1$, $\pi_2$ be as in Lemma \ref{Lm:Atp} and recall that $\pi_1$ is onto $G$, $\ker{\pi_1}$ is isomorphic to a subloop of $\rnuc{Q,\cdot}$, and $\rnuc{Q,\cdot}=1$ by Theorem \ref{Th:RightNucleus}. Therefore $N\cong G$ and we are done by Lemma \ref{Lm:NSylow}.

(ii) Let $P$ be the unique Sylow $p$-subgroup of $G$. Let $a$, $b\in (Q,\cdot)$ be such that $|a|=q$ and $|b|=p$. Then $G=\langle R_a,\,R_b\rangle$ by Corollary \ref{Cr:Factorizations}. Since $|R_a|=q$ and $|R_b|=p$, we have $G=\langle R_a\rangle P$ and $|G|=p^k q$ with $1\le k\le 3$. The
case $|G|=pq$ leads to $Q\cong\Z_{pq}$, a contradiction.
\end{proof}

\begin{theorem}\label{Th:LNuc}
Let $p>q$ be odd primes and let $Q$ be a nonassociative right Bol loop of order $pq$. Then $\lnuc{Q}\cong \Z_p$.
\end{theorem}
\begin{proof}
By Proposition \ref{Pr:RMltBol}, $G=\rmlt{Q}$ is of order $p^2q$ or $p^3q$. Since $|G|=|Q|\cdot |\rinn{Q}|$, it follows that $\rinn{Q}$ is a $p$-group. Consequently, every orbit of $\rinn{Q}$ has a size that is a power of $p$. Since $1$ is a fixed point of $\rinn{Q}$, there must be at least $p-1$ additional fixed points of $\rinn{Q}$. Now, $\lnuc{Q}$ consists precisely of the fixed points of $\rinn{Q}$.

We have shown that $|\lnuc{Q}|\ge p$, so either $|\lnuc{Q}|=p$ and we are done, or $|\lnuc{Q}|=pq$, $\lnuc{Q}=Q$, a contradiction.
\end{proof}

\subsection{The complete mappings $\theta_i$ are linear}

Niederreiter and Robinson obtained a number of results for right Bol loops of order $pq$ for which every complete mapping $\theta_i$ of \eqref{Eq:Mult} is linear. We now prove that every $\theta_i$ must be linear should \eqref{Eq:Mult} yield a right Bol loop.

\begin{lemma}\label{Lm:AuxLNuc}
Let $p>q$ be odd primes, and let $Q$ be a groupoid defined on $\Z_q\times \Z_p =\langle a\rangle\times\langle b\rangle$ by \eqref{Eq:Mult}, where every $\theta_i$ is a complete mapping of $\Z_p$. Then $b^j\cdot a^kb^\ell = b^ja^k\cdot b^\ell$ holds for every $j$, $\ell$ if and only if $\theta_k$ is linear.
\end{lemma}
\begin{proof}
First note that \eqref{Eq:Mult} implies $a^ib^k\cdot b^\ell = a^ib^{k+\ell}$ for every $i$, $k$, $\ell$. Now fix $j$, $k$, $\ell$ and let $m$ be such that $m+m\theta_k=\ell$. Then $b^j\cdot a^k b^\ell = a^kb^{m+(j+m)\theta_k}$, while $b^ja^k\cdot b^\ell = a^kb^{j\theta_k}\cdot b^\ell = a^kb^{j\theta_k+\ell}$. Therefore $b^j\cdot a^kb^\ell = b^ja^k\cdot b^\ell$ holds if and only if $m+(j+m)\theta_k = j\theta_k+\ell = j\theta_k + m+ m\theta_k$, or, equivalently, $(j+m)\theta_k = j\theta_k+m\theta_k$. Note that as $\ell$ ranges over $\Z_p$, so does $m$.
\end{proof}

\begin{theorem}\label{Th:Linear}
Let $p>q$ be odd primes and let $Q$ be a right Bol loop of order $pq$. Then all complete mappings $\theta_i$ of \eqref{Eq:Mult} are linear.
\end{theorem}
\begin{proof}
By Theorem \ref{Th:LNuc}, $\lnuc{Q}$ contains a subloop of order $p$. Since $\langle b\rangle$ accounts for all elements of order $p$ in $Q$, we have $\langle b\rangle\le\lnuc{Q}$. We are done by Lemma \ref{Lm:AuxLNuc}.
\end{proof}

\section{The right Bol loops of order $pq$ up to isomorphism}

In this section we present a classification of right Bol loops of order $pq$ up to isomorphism.

\subsection{An abstract construction}

The following is motivated by Lemma \ref{Lm:BasicThetaProperties} and Proposition \ref{Pr:MultiplicationFormula}.

Let $p>q$ be odd primes, and let $\Theta = \setof{\theta_i}{i\in \Z_q}$ be a collection of complete mappings of $\Z_p$ such that $\theta_0=1$ and $0\theta_i=0$ for every $i\in \Z_q$. Define multiplication on $\Z_q\times \Z_p$ by
\[
    (i,j)(k,\ell) = (i+k,\ m+(j+m)\theta_i^{-1}\theta_{i+k}),
\]
where $m+m\theta_k=\ell$.

We have $(0,0)(k,\ell) = (k,m+m\theta_0^{-1}\theta_k) = (k,m+m\theta_k)$, where $m+m\theta_k=\ell$, so $(0,0)(k,\ell)=(k,\ell)$. Similarly, $(i,j)(0,0) = (i,n+(j+n)\theta_i^{-1}\theta_i) = (i,2n+j)$, where $2n = n+n\theta_0=0$, so $(i,j)(0,0)=(i,j)$. Therefore $(0,0)$ is the identity element.

We claim that all right translations biject. Given $(k,\ell)$, $(u,v)$, we need to find $(i,j)$ such that $(i,j)(k,\ell) = (i+k,m+(j+m)\theta_i^{-1}\theta_{i+k}) = (u,v)$, where $m+m\theta_k=\ell$ is determined by $k$, $\ell$. We must take $i=u-k$, and we need a $j$ such that $m+(j+m)\theta_i^{-1}\theta_u = v$, that is, $j = (v-m)\theta_u^{-1}\theta_i - m$.

However, the left translations are not necessarily bijections. We claim that all left translations are bijections if and only if $\theta_i^{-1}\theta_j$ is a complete mapping for every $i$, $j\in \Z_q$. Indeed, given $(i,j)$, $(u,v)$, we need to find $(k,\ell)$ such that $(i,j)(k,\ell) = (i+k,m+(j+m)\theta_i^{-1}\theta_{i+k}) = (u,v)$, where $m+m\theta_k=\ell$. We must take $k=u-i$, and we want $\ell$ such that $-j+(m+j)+(m+j)\theta_i^{-1}\theta_u = v$, where $m+m\theta_k=\ell$. Now, as $\ell$ ranges over $\Z_q$, so does $m$, so we need an $m$ such that $(m+j)+(m+j)\theta_i^{-1}\theta_u = v+j$. We will be always able to find such an $m$ if and only if $\theta_i^{-1}\theta_u$ is a complete mapping.

\medskip

Therefore, for odd primes $p>q$, let $\Theta=\setof{\theta_i}{i\in \Z_q}$ be a collection of complete mappings such that $\theta_0=1$, $0\theta_i=0$ for every $i\in \Z_q$, and $\theta_i^{-1}\theta_j$ is a complete mapping for every $i$, $j\in \Z_q$. Then, and only then, will we define $\mathcal Q(\Theta)$ on $\Z_q\times \Z_p$ by
\begin{equation}\label{Eq:MultAbstract}
    (i,j)(k,\ell) = (i+k,\ m+(j+m)\theta_i^{-1}\theta_{i+k}),
\end{equation}
where $m+m\theta_k=\ell$. We have proved above that $\mathcal Q(\Theta)$ is a loop.

We would like to know when $\mathcal Q(\Theta)$ is a right Bol loop. This problem was resolved by Niederreiter and Robinson when $\Theta$ consists of linear complete mappings. We can restate their results as follows:

\begin{theorem}[{{compare \cite[Theorems 6, 8 and 11]{NiederreiterRobinson_pq}}}]\label{Th:NiRoBig}
Let $p>q$ be odd primes. Let $\mathcal Q(\Theta)$ be the loop defined by \eqref{Eq:MultAbstract}, and suppose that $\theta_i$ is linear for every $i\in \Z_q$. Then $\mathcal Q(\Theta)$ is a right Bol loop if and only if there exists a bi-infinite sequence $\{u_i\}$ with $u_i\in\F_p$ solving the recurrence relation
\begin{equation}\label{Eq:Recurrence}
    u_{n+2} = \lambda u_{n+1} - u_n,
\end{equation}
for some $\lambda\in\F_p^*$, and we have $u_0=1$, $u_i^{-1}u_j\in\F_p\setminus\{0,-1\}$ for every $i$, $j$, and $\theta_i = u_k^{-1}$ whenever $i\equiv k\pmod q$.

Suppose that two right Bol loops correspond to the bi-infinite sequences $\{u_i\}$ and $\{v_i\}$, respectively. Then the loops are isomorphic if and only if there is $0\ne s\in \Z_q$ such that $u_i = v_{si}$ for every $i\in \Z_q$, and the loops are isotopic if and only if there are $0\ne s\in \Z_q$ and $r\in\Z_q$ such that $u_i=v_r^{-1}v_{si+r}$ for every $i\in\Z_q$.

Finally, the obtained right Bol loop is a right Bruck loop if and only if $u_i=u_{-i}$ for all $i\in \Z_q$.
\end{theorem}
\begin{proof}
Theorem 6 of \cite{NiederreiterRobinson_pq} applies as long as we verify the following properties in our abstract loop $\mathcal Q(\Theta)$: $(i,0)(k,0)=(i+k,0)$, $(0,j)(0,\ell)=(0,j+\ell)$, $(i,0)(0,\ell) = (i,\ell)$, $(0,j)(i,0) = (i,j\theta_i)$, and $\setof{(0,j)}{j\in \Z_p}$ is a normal subloop of $\mathcal Q(\Theta)$. These conditions are routinely verified from \eqref{Eq:MultAbstract}, where we note that $\setof{(0,j)}{j\in \Z_p}$ is the kernel of the homomorphism $(i,j)\mapsto i$.

(Our assumption on $\Theta$ that $\theta_i^{-1}\theta_j$ is complete for every $i$, $j$ is somewhat suppressed in \cite[Theorem 6]{NiederreiterRobinson_pq}, perhaps being implicit in the fact that their coefficient $R(j,\ell) = (\theta_\ell+1)^{-1}(\theta_{j+\ell}\theta_j^{-1}+1)$ must be invertible. In any case, we have seen that this assumption is necessary in order to obtain a loop, so we must also enforce its counterpart $u_i^{-1}u_j\not\in\{0,-1\}$ for the solutions of the recurrence relation.)

The isomorphism and isotopism problems for bi-infinite sequence are solved in \cite[Theorem 11]{NiederreiterRobinson_pq}. The claim about right Bruck loops is \cite[Theorem 8]{NiederreiterRobinson_pq}.
\end{proof}

Theorem \ref{Th:NiRoBig} in combination with Theorem \ref{Th:Linear} allows us to attack the classification of right Bol loops of order $pq$.

Given $u =\{u_i\}$, $v=\{v_i\}$, write
\begin{equation}\label{Eq:Equivalence}
    u\sim v
\end{equation}
if and only if there is $0\ne s\in \Z_q$ such that $u_i = v_{si}$ for every $i$.

\subsection{The eigenproblem for a circulant matrix}

We will solve the recurrence relation \eqref{Eq:Recurrence} in several steps. First, since we demand $\theta_i = u_k^{-1}$ whenever $i\equiv k\pmod q$, the solution $\{u_i\}$ must be periodic with period $q$. Moreover, the constant $0\ne \lambda\in \F_p$ is part of the problem.

Let $A=(a_{i,j})$ be an $n\times n$ matrix with rows and columns indexed by elements of $\Z_n$. Then $A$ is a \emph{circulant matrix} if $a_{(i+1)\;\textrm{mod}\;q,\ (j+1)\;\textrm{mod}\;q} = a_{i,j}$ for every $i$, $j\in\Z_n$. It is clear that a circulant matrix is determined by its first row.

Let $P$ be the permutation matrix corresponding to the $q$-cycle $(1,2,\dots,q)$, and let $A = P+P^{-1}$. Then $A$ is a $q\times q$ circulant matrix with first row equal to $(0,1,0,\dots,0,1)$. Once we rewrite the recurrence \eqref{Eq:Recurrence} as $u_{n+2}-\lambda u_{n+1}+u_n=0$, we see that solving \eqref{Eq:Recurrence} with period $q$ for some constant $\lambda$ is equivalent to solving the eigenvalue problem $Au=\lambda u$.

We will follow the standard approach to real circulant matrices \cite{Davis}. Since we will need to work in $\F_{p^2}$, let us first recall some basic facts about quadratic extensions.

\medskip

Let $p$ be an odd prime, and let $t$ be any element of $\F_p$ that is not a square modulo $p$. Then $\F_{p^2}$ can be represented as $\setof{u+v\sqrt{t}}{u,\,v\in \F_p}$ with addition $(u_1+v_1\sqrt{t})+(u_2+v_2\sqrt{t}) = (u_1+u_2) + (v_1+v_2)\sqrt{t}$ and multiplication $(u_1+v_1\sqrt{t})(u_2+v_2\sqrt{t}) = (u_1u_2+v_1v_2t) + (u_1v_2+v_1u_2)\sqrt{t}$.

Let $\alpha=u+v\sqrt{t}\in \F_{p^2}$. The conjugate of $\alpha$ is the element $\alpha^* = u-v\sqrt{t}$. The norm of $\alpha$ is given by $N(\alpha) = \alpha\alpha^* = u^2-v^2t$, giving rise to a multiplicative map $N:\F_{p^2}\to \F_p$. The trace of $\alpha$ is given by $\mathrm{tr}(\alpha) = \alpha + \alpha^* = 2u$, giving rise to an additive map $\mathrm{tr}:\F_{p^2}\to \F_p$. The characteristic polynomial of $\alpha$ is the polynomial $x^2 - \mathrm{tr}(\alpha)x + N(\alpha)$ over $\F_p$, which has both $\alpha$ and $\alpha^*$ as roots.

\medskip

\begin{lemma}\label{Lm:RootSums}
Let $p>q$ be odd primes such that $q$ divides $p^2-1$, and let $\omega$ be a primitive $q$th root of unity in $\F_{p^2}$. Then $\omega + \omega^{-1}\in \F_p$.
\end{lemma}
\begin{proof}
There is nothing to prove when $q$ divides $p-1$ since then $\omega\in \F_p$. Suppose that $q$ does not divide $p-1$. Since $\omega$ is a (primitive) $q$th root of unity in $\F_{p^2}$, we have $1=N(1)=N(\omega^q)=N(\omega)^q$. Since $q$ does not divide $p-1$, it follows that $N(\omega)=1$, and with $\omega = x+y\sqrt{t}$ we get $\omega^{-1} = x - y\sqrt{t}$ and $\omega + \omega^{-1} = 2x\in \F_p$.
\end{proof}

\begin{lemma}\label{Lm:Eigenvalues}
Let $p>q$ be odd primes such that $q$ divides $p^2-1$. Let $A$ be the $q\times q$ circulant matrix with first row equal to $(0,1,0,\dots,0,1)$. Let $\omega$ be a primitive $q$th root of unity in $\F_{p^2}$. For $0\le i<q$, let $\lambda_i = \omega^i+\omega^{-i}$ and $v_i = (1,\omega^i,\omega^{2i},\dots,\omega^{(q-1)i})^T$.
\begin{enumerate}
\item[(i)] For every $0\le i<q$, $\lambda_i$ is an element of the prime field of $\F_{p^2}$.
\item[(ii)] For every $0\le i<q$, $\lambda_i$ is an eigenvalue of $A$ over $\F_{p^2}$ with eigenvector $v_i$.
\item[(iii)] For $0\le i<j<q$, $\lambda_i=\lambda_j$ if and only if $i+j\equiv 0\pmod q$. In particular, $\lambda_0$ has multiplicity $1$, and every $\lambda_i$ with $1\le i\le (q-1)/2$ has multiplicity $2$.
\item[(iv)] For $0< i\le (q-1)/2$, the eigenvectors $v_i$, $v_{q-i}$ are linearly independent.
\end{enumerate}
\end{lemma}
\begin{proof}
Part (i) follows from Lemma \ref{Lm:RootSums}. For (ii), let us index the rows and columns of $A$ by elements of $\Z_q$, and let $e_j$ be the row vector whose only nonzero entry is $1$ in position $j$. The $j$th row of $A-\lambda_iI$ is equal to $e_{j-1} - \lambda_i e_j + e_{j+1}$, where the subscripts of $e$ are taken modulo $q$. We have $(e_{j-1}-\lambda_ie_j + e_{j+1})\cdot v_i^T = w^{(j-1)i} -(w^i+w^{-i})w^{ji} + w^{(j+1)i} = 0$.

(iii) We clearly have $\lambda_i=\lambda_{q-i}$. Suppose that there are $0\le i<j\le q-1$ such that $\lambda_i=\lambda_j=\lambda$. Then $\omega^i$, $\omega^{-i}$ and $\omega^j$ are roots of the characteristic polynomial $x^2-\lambda x + 1$. Therefore $j\in\{i,\,-i\}$.

(iv) The vectors $v_i$, $v_{q-i}$ have the same first coordinates but different second coordinates, so they are linearly independent.
\end{proof}

\subsection{Identifying valid solutions}

Suppose that $q$ divides $p^2-1$. We have now solved the recurrence relation \eqref{Eq:Recurrence} in $\F_{p^2}$ subject only to the restriction that it be periodic with period $q$. While the eigenvalues happen to lie in $\F_p$, the eigenvectors need not have all components in $\F_p$ when $q$ divides $p+1$.

Let us leave the solution for $\lambda_0=2$ aside for now. By Lemma \ref{Lm:Eigenvalues}, the general solution for the eigenvalue $\lambda_i$ with $i>0$ is of the form $u = \gamma v_i + \delta v_{-i}$, where $\gamma$, $\delta\in\F_{p^2}$.

Following Theorem \ref{Th:NiRoBig}, we are only interested in solutions $u$ such that $u_0=1$ and $u_i^{-1}u_j\in\F_p\setminus\{0,-1\}$. Finally, we need to understand the solutions modulo the equivalence \eqref{Eq:Equivalence}. We will deal with all these requirements at the same time.

The final solutions will eventually be described by a certain set $\Gamma$, a subset of
\[
    \Gamma''=\left\{\begin{array}{ll}
        \F_p,&\text{ if $q$ divides $p-1$},\\
        1/2 + \F_p\sqrt{t},&\text{ if $q$ divides $p+1$}.
    \end{array}\right.
\]

\begin{lemma}\label{Lm:ConvexRootSums}
Let $p>q$ be odd primes such that $q$ divides $p^2-1$, and let $\omega$ be a primitive $q$th root of unity in $\F_{p^2}$. Let $\gamma\in \F_{p^2}$. Then $\gamma\omega + (1-\gamma)\omega^{-1}\in \F_p$ if and only if $\gamma\in\Gamma''$.
\end{lemma}
\begin{proof}
Let $\omega = x+y\sqrt{t}$ and $\gamma = u + v\sqrt{t}$. Suppose that $q$ divides $p-1$, so $\omega = x\in \F_p$. Then $\gamma\omega + (1-\gamma)\omega^{-1} = (u+v\sqrt{t})x + ((1-u)-v\sqrt{t})x^{-1}$ is of the form $r+s\sqrt{t}$, where $s = vx-vx^{-1}$. We see that $s=0$ if and only if $v=0$ or $x-x^{-1}=0$, but the latter condition cannot occur since $x^2 = \omega^2\ne 1$.

Now suppose that $q$ divides $p+1$, so $\omega = x+y\sqrt{t}$ for some $y\ne 0$. Then $\gamma\omega + (1-\gamma)\omega^{-1} = (u+v\sqrt{t})(x+y\sqrt{t}) + ((1-u)-v\sqrt{t})(x-y\sqrt{t})$ is of the form $r+s\sqrt{t}$, where $s = uy+vx -(1-u)y-vx = (2u-1)y$. Since $y\ne 0$, we conclude that $s=0$ if and only if $2u-1\equiv 0\pmod p$.
\end{proof}

\begin{lemma}\label{Lm:Equivalence}
Let $0<j$, $k\le q-1$. Then for every $\gamma$, $\delta\in \F_{p^2}$ we have $\gamma v_j + \delta v_{-j} \sim \gamma v_k + \delta v_{-k}$.
\end{lemma}
\begin{proof}
Let $s = s(j,k) = jk^{-1}\pmod q$, and note that $s(j,k) = s(-j,-k)$. Then for every $0\le i\le q-1$ we have $v_{j,i} = \omega^{ij} = \omega^{ijk^{-1}k} = \omega^{(is)k} = v_{k,si}$ and, similarly $v_{-j,i} = v_{-k,si}$. Therefore $(\gamma v_j+\delta v_{-j})_i = \gamma v_{j,i} + \delta v_{-j,i} = \gamma v_{k,si} + \delta v_{-k,si} = (\gamma v_k+\delta v_{-k})_{si}$.
\end{proof}

\begin{lemma}
Let $0<j$, $k\le q-1$. Let $u_j = \{u_{j,i}\}$ be a solution to $Au_j = \lambda_ju_j$ such that $u_{j,i}\in \F_p\setminus\{0,-1\}$ and $u_{j,0}=1$. Then there is a solution $u_k = \{u_{k,i}\}$ to $Au_k=\lambda_ku_k$ such that $u_{k,i}\in \F_p\setminus\{0,-1\}$, $u_{k,0}=1$ and $u_j\sim u_k$.
\end{lemma}
\begin{proof}
Since $Au_j=\lambda_ju_j$, we have $u_j = \gamma v_j + \delta v_{-j}$ for some $\gamma$, $\delta\in \F_{p^2}$. We have $u_{j,0}=1$ by assumption and $v_{j,0} = v_{-j,0}=1$ always, which forces $\delta = 1-\gamma$. Lemma \ref{Lm:ConvexRootSums} therefore applies.

Let $u_k=\gamma v_k+(1-\gamma)v_{-k}$. Obviously, $u_k$ solves $Au_k = \lambda_ku_k$. Moreover, Lemma \ref{Lm:ConvexRootSums} implies that $u_{k,i}\in \F_p$ for every $i$ and $u_{k,0} = u_{j,0}=1$. By Lemma \ref{Lm:Equivalence}, $u_j\sim u_k$. Since the coordinates of $u_k$ are merely the permuted coordinates of $u_j$, we also get $u_{k,i}\not\in\{0,-1\}$.
\end{proof}

We are therefore interested in the solutions
\[
    u(\gamma) = \gamma v_1 + (1-\gamma)v_{-1}
\]
to $Au = \lambda_1 u$, where $\gamma\in\Gamma''$.

\begin{lemma}\label{Lm:Silly}
Let $F$ be a field and $x$, $y\in F$ such that $y^2\ne -1$. Then $xy+(1-x)y^{-1}=0$ if and only if $xy^2+(1-x)y^{-2}=-1$.
\end{lemma}
\begin{proof}
Note that $x y^2 + 1 + (1 - x) y^{-2}   =   ( x y + (1 - x) y^{-1}  )( y + y^{-1} )$.
\end{proof}

Set $\Gamma' = \setof{\gamma\in\Gamma''}{1-\gamma^{-1}\not\in\langle\omega\rangle}$.

\begin{lemma}
The solutions to $Au = \lambda_1u$ such that $u_i\in \F_p\setminus\{0,-1\}$ and $u_0=1$ are precisely the $p-(q-1)$ solutions $\setof{u(\gamma)}{\gamma\in\Gamma'}$.

Moreover, $u(1/2)$ is a solution, and $u(\gamma)$ is a solution if and only if $u(1-\gamma)$ is a solution.
\end{lemma}
\begin{proof}
A vector $u$ solves $Au=\lambda_1u$ and satisfies $u_0=1$ if and only if $u=\gamma v_1 + (1-\gamma) v_{-1}$. Every $\omega^i$ with $1\le i\le q-1$ is a primitive $q$th root of unity in $\F_{p^2}$, so Lemma \ref{Lm:ConvexRootSums} shows that $u_i\in \F_p$ for every $i$ if and only if $\gamma\in\Gamma''$.

By Lemma \ref{Lm:Silly}, $u_i=0$ for some $i$ if and only if $u_{2i}=-1$. It therefore suffices to investigate the condition $u_i=0$. Call $\gamma\in \F_{p^2}$ \emph{bad} if $u_i\in\{0,-1\}$ for some $i$. If $\gamma$ is bad then $\gamma\omega^i+(1-\gamma)\omega^{-i}\in \{0,-1\}\subseteq \F_p$ for some $i$, and thus $\gamma\in \Gamma''$ by Lemma \ref{Lm:ConvexRootSums}. Moreover, $\gamma=0$ is not bad (since $\omega$ has odd order).

The following conditions are therefore equivalent: $\gamma$ is bad, $\gamma\omega^i = (\gamma-1)\omega^{-i}$ for some $i$, $\gamma\omega^{2i} = \gamma-1$ for some $i$, $1-\gamma^{-1}=\omega^{2i}$ for some $i$. Thus $\gamma$ is bad if and only if $1-\gamma^{-1}\in\langle\omega\rangle$. The mapping $x\mapsto 1-x^{-1}$ is a bijection $\F_{p^2}\setminus\{0\}\to \F_{p^2}\setminus\{1\}$, hence precisely $q-1$ values of $\gamma\in\Gamma''$ are bad.

We have now established that the solutions are precisely the elements of $\setof{u(\gamma)}{\gamma\in\Gamma'}$. We note that $1/2\in \Gamma''$ is not bad because $1-(1/2)^{-1} = 1-2 = -1\not\in\langle\omega\rangle$. Note that if $\gamma\in\Gamma''$ then $1-\gamma\in\Gamma''$. Also note that if $1-\gamma^{-1} = \omega^j$ for some $j$, then $1-(1-\gamma)^{-1} = \omega^{-j}$. Thus, if $u(\gamma)$ is a solution then $\gamma\in\Gamma''$ implies $1-\gamma\in\Gamma''$, which in turn means that $u(1-\gamma)$ is a solution.
\end{proof}

\begin{lemma}\label{Lm:Mirror}
Let $\gamma$, $\delta\in \F_{p^2}$. Then $u(\gamma)\sim u(\delta)$ if and only if $\gamma=\delta$ or $\gamma = 1-\delta$.
\end{lemma}
\begin{proof}
If $\gamma=\delta$ then obviously $u(\gamma)\sim u(\delta)$. Also, for every $i$ we have $u(\gamma)_i = \gamma\omega^i + (1-\gamma)\omega^{-i} = u(1-\gamma)_{-i}$, so $u(\gamma)\sim u(1-\gamma)$ via $s=-1$.

Suppose that there is $0\ne s\in \Z_q$ such that $u(\gamma)_i = u(\delta)_{si}$ for every $i$. In particular,
\begin{align*}
    \gamma\omega + (1-\gamma)\omega^{-1} &= \delta\omega^s + (1-\delta)\omega^{-s},\\
    \gamma\omega^{-1} + (1-\gamma)\omega &= \delta\omega^{-s}+(1-\delta)\omega^s,
\end{align*}
and summing up these equations yields $\omega+\omega^{-1} = \omega^s + \omega^{-s}$. It follows from Lemma \ref{Lm:Eigenvalues} that $s=\pm 1$. The choice $s=1$ yields $\gamma=\delta$, and the choice $s=-1$ yields $\gamma = 1-\delta$.
\end{proof}

Finally, set
\[
        \Gamma = \left\{\begin{array}{ll}
             \setof{\gamma\in\Gamma'}{1\le \gamma\le (p+1)/2},&\text{if $q$ divides $p-1$},\\
             \setof{\gamma = 1/2+m\sqrt{t}\in\Gamma'}{0\le m\le (p-1)/2},&\text{if $q$ divides $p+1$}.
        \end{array}\right.
\]
Note that $\Gamma$ is a set of representatives of the equivalence classes of $\sim$ on $\Gamma'$.

\begin{lemma}
If $u=u(\gamma)$ for some $\gamma\in\Gamma$, then $u_i^{-1}u_j\ne -1$ for every $i$, $j\in \Z_q$.
\end{lemma}
\begin{proof}
Suppose that $u_iu_j^{-1}=-1$ for some $i$, $j$. Then $u_i=-u_j$, so $\gamma\omega^i + (1-\gamma)\omega^{-i} = - \gamma\omega^j - (1-\gamma)\omega^{-j}$. Note that we have $\gamma\ne 0$. Dividing by $\gamma$, we get $\omega^i + (\gamma^{-1}-1)\omega^{-i} = -\omega^j - (\gamma^{-1}-1)\omega^{-j}$. Solving for $1-\gamma^{-1}$, we get
\[
    1-\gamma^{-1} = \frac{\omega^i+\omega^j}{\omega^{-i}+\omega^{-j}} = \omega^{i+j}\in \langle\omega\rangle,
\]
a contradiction.
\end{proof}

We can now finish the proof of Theorem \ref{Th:Main}:

Let $p>q$ be odd primes. Focusing on nonassociative right Bol loops $Q$ of order $pq$, we can assume that $q$ divides $p^2-1$ by Theorem \ref{Th:BolSubloop}. By Lemma \ref{Lm:BasicThetaProperties}, Proposition \ref{Pr:MultiplicationFormula} and Theorem \ref{Th:Linear}, $Q=\mathcal Q(\Theta)$, where every $\theta_i$ is linear, $\theta_0=1$, and $\theta_i^{-1}\theta_j\ne\{0,-1\}$ for every $i$, $j$. By Theorem \ref{Th:NiRoBig}, we instead solve the recurrence relation \eqref{Eq:Recurrence} with period $q$ subject to the conditions $u_0=1$, $u_i^{-1}u_j\in\F_p\setminus\{0,-1\}$, and modulo the equivalence $\sim$.

To obtain period $q$ in the solution, we solve the system $Au = \lambda u$, where $A$ is the $q\times q$ circulant matrix with first row $(0,1,0,\dots,0,1)$. The eigenvalues and eigenvectors are given by Lemma \ref{Lm:Eigenvalues}. The eigenvalue $\lambda_0=2$ has eigenvector $v_0=(1,1,\dots,1)$, which is a valid solution, and this solution cannot be equivalent to any other solution modulo $\sim$ since it is the only solution where all coordinates are the same. It is obvious from \eqref{Eq:MultAbstract} that this solution yields the cyclic group $\Z_{pq}$.

We showed above that up to $\sim$ it suffices to consider solutions for the eigenvalue $\lambda_1$. These solutions are precisely the elements of $S=\setof{u(\gamma)}{\gamma\in\Gamma'}$, we have $u(1/2)\in S$, and if $u(\gamma)\in S$ then $u(1-\gamma)\in S$. The involutory action $u(\gamma)\mapsto u(1-\gamma)$ of Lemma \ref{Lm:Mirror} therefore restricts to $S$. The unique fixed point of the action is $u(1/2)\in S$. The remaining $p-(q-1)-1 = p-q$ points of $S$ are paired up modulo $\sim$.

We have obtained $1 + 1 + (p-q)/2 = (p-q+4)/2$ solutions up to $\sim$.

When $q$ divides $p-1$, the solution $u(1)$ yields a (nonabelian) group by Corollary \ref{Cr:LinearAssociative}. The solution $u(1/2)$ yields a right Bruck loop by Theorem \ref{Th:NiRoBig}. Since we have already accounted for all groups of order $pq$, this right Bruck loop must be nonassociative. All other solutions yield nonassociative non-Bruck right Bol loops, by Theorem \ref{Th:Uniqueness}.

\section{Open problems}

\begin{problem}
Let $p>q$ be odd primes such that $q$ divides $p^2-1$, and let $Q$ be a nonassociative right Bol loop of order $pq$. Is the order of the right multiplication group of $Q$ equal to $p^2q$?
\end{problem}

\begin{problem}
Let $p>q$ be odd primes such that $q$ divides $p^2-1$. Classify right Bol loops of order $pq$ up to isotopism.
\end{problem}

\begin{conjecture}\label{Cj:3p}
Let $p>3$ be a prime. Then the number of right Bol loops of order $3p$ up to isotopism is equal to $\lfloor (p+5)/6\rfloor + 1$.
\end{conjecture}

We have verified Conjecture \ref{Cj:3p} for all primes $p$ less than $1000$ by means of Theorem \ref{Th:NiRoBig}.

\begin{problem}
Let $p>q$ be odd primes. Classify right Bol loops of order $p^2q$.
\end{problem}


\begin{thebibliography}{99}
\bibitem{Aschbacher2006} Aschbacher, M. On Bol loops of exponent 2. J. Algebra \textbf{288} (2005), no. 1, 99--136.

\bibitem{BaumeisterStein}  Baumeister, B. and Stein, A., \emph{Self-invariant $1$-factorizations of complete graphs and finite Bol loops of exponent $2$}, Beitr\"age Algebra Geom. \textbf{51} (2010), no. \textbf{1}, 117--135.

\bibitem{Bol} Bol, G., \emph{Gewebe und Gruppen}, Math. Ann. \textbf{114} (1937), no. \textbf{1}, 414--431.

\bibitem{Bruck} Bruck, R.H., \emph{A survey of binary systems}, Ergebnisse der Mathematik und ihrer Grenzgebiete, Springer Verlag, Berlin-G\"ottingen-Heidelberg, 1958.

\bibitem{BurnI} Burn, R.P., \emph{Finite Bol loops}, Math. Proc. Cambridge Philos. Soc. \textbf{84} (1978), no. \textbf{3}, 377--385.

\bibitem{BurnII} Burn, R.P., \emph{Finite Bol loops: II}, Math. Proc. Cambridge Philos. Soc. \textbf{89} (1981), no. \textbf{3}, 445--455.

\bibitem{BurnIII} Burn, R.P., \emph{Finite Bol loops: III}, Math. Proc. Cambridge Philos. Soc. \textbf{97} (1985), no. \textbf{2}, 219--223. (See also Burn, R.~P., \emph{Corrigenda: ``Finite Bol loops: III''}, Math. Proc. Cambridge Philos. Soc. \textbf{98} (1985), no. \textbf{2}, 381.)

\bibitem{Davis} David, P.J., \emph{Circulant matrices}, Wiley, New York, 1970.

\bibitem{FoKiPh} Foguel, T., Kinyon, M.K. and Phillips, J.D., \emph{On twisted subgroups and Bol loops of odd order}, Rocky Mountain J. Math. \textbf{36} (2006), no. \textbf{1}, 183--212.

\bibitem{GAP} The GAP Group, GAP -- Groups, Algorithms, and Programming, Version 4.5.5; 2012. (http://www.gap-system.org)

\bibitem{GlaubermanI} Glauberman, G., \emph{On loops of odd order}, J. Algebra \textbf{1} (1964), 374--396.


\bibitem{Hall_1959} Hall, M. Jr. \emph{The theory of groups.} The Macmillan Co., New York, N.Y. 1959.

\bibitem{Huppert} Huppert, B. \emph{Endliche Gruppen. I.} Die Grundlehren der Mathematischen Wissenschaften, Band \textbf{134}, Springer-Verlag, Berlin-New York, 1967.

\bibitem{Liebeck} Liebeck, M.W., \emph{The classification of finite simple Moufang loops}, Math. Proc. Cambridge Philos. Soc. \textbf{102} (1987), no. \textbf{1}, 33--47.

\bibitem{Moufang} Moufang, R., \emph{Zur Struktur von Alternativk\"orpern}, Math. Ann. \textbf{110} (1935), no. \textbf{1}, 416--430.

\bibitem{Nagy_invariants} Nagy, G.P., \emph{Group invariants of certain Burn loop classes}, Finite geometry and combinatorics (Deinze, 1997). Bull. Belg. Math. Soc. Simon Stevin \textbf{5} (1998), no. \textbf{2}--\textbf{3}, 403--415.

\bibitem{Nagy_simple2} Nagy, G.P., \emph{A class of finite simple Bol loops of exponent $2$}, Trans. Amer. Math. Soc. \textbf{361} (2009), no. \textbf{10}, 5331--5343.

\bibitem{Nagy_simple} Nagy, G.P., \emph{A class of simple proper Bol loops}, Manuscripta Math. \textbf{127} (2008), no. 1, 81--88.

\bibitem{LOOPS} Nagy, G.P. and Vojt\v{e}chovsk\'y, P., \texttt{LOOPS}: Computing with quasigroups and loops in GAP, version 3.1.0, available at \texttt{www.math.du.edu/loops}.

\bibitem{NiederreiterRobinson_pq} Niederreiter, H. and Robinson, K.H., \emph{Bol loops of order $pq$}, Math. Proc. Cambridge Philos. Soc. \textbf{89} (1981), no. \textbf{2}, 241--256.

\bibitem{NiederreiterRobinson_3p} Niederreiter, H. and Robinson, K.H., \emph{On isomorphisms and isotopisms of Bol loops of order $3p$}, Comm. Algebra \textbf{22} (1994), no. \textbf{1}, 345--347.

\bibitem{Paige} Paige, L.J., \emph{A class of simple Moufang loops}, Proc. Amer. Math. Soc. \textbf{7} (1956), 471--482.

\bibitem{Robinson_thesis} Robinson, D.A., \emph{Bol loops},  Ph.D. Thesis, The University of Wisconsin-Madison, 1964, 83 pp.

\bibitem{Robinson} Robinson, D.A., \emph{Bol loops}, Trans. Amer. Math. Soc. \textbf{123} (1966), 341--354.

\bibitem{Sabinin} Sabinin, L.V. \emph{Smooth quasigroups and loops.} Mathematics and its Applications, \textbf{492}. Kluwer Academic Publishers, Dordrecht, 1999.

\bibitem{Sharma_18} Sharma, B.L., \emph{Classification of Bol loops of order 18}, Acta Univ. Carolin. Math. Phys. \textbf{25} (1984), no. \textbf{1}, 37--44.

\bibitem{Sharma_normal} Sharma, B.L., \emph{Normal subloop of a finite Bol loop of order $pq$}, Proc. Nat. Acad. Sci. India Sect. A \textbf{57} (1987), no. \textbf{2}, 136--141.

\bibitem{Sharma_pq} Sharma, B.L., \emph{Bol loops of order $pq$ with $q\nmid (p^2-1)$}, Boll. Un. Mat. Ital. A (\textbf{7}) \textbf{1} (1987), no. \textbf{2}, 163--169.

\bibitem{SharmaSolarin_2p2} Sharma, B.L. and Solarin A.R.T., \emph{Finite Bol loops of order $2p^2$}, Simon Stevin \textbf{60} (1986), no. \textbf{2}, 133--156.

\bibitem{SharmaSolarin_3p} Sharma, B.L. and Solarin A.R.T., \emph{On the classification of Bol loops of order $3p$ ($p>3$)}, Comm. Algebra \textbf{16} (1988), no. \textbf{1}, 37-55.

\bibitem{Ungar} Ungar, A.A., \emph{Beyond the Einstein addition law and its gyroscopic Thomas precession. The theory of gyrogroups and gyrovector spaces.} Fundamental Theories of Physics \textbf{117}, Kluwer Academic Publishers Group, Dordrecht, 2001.

\bibitem{Vojtechovsky} Vojt\v{e}chovsk\'y, P., \emph{Three lectures on automorphic loops}, Proceedings of Workshops Loops '15, Ohrid, Macedonia, published in Quasigroups Related Systems \textbf{23} (2015), no. \textbf{1}, 129--163.

\end{thebibliography}
\end{document}